\documentclass[12pt]{amsart}
\usepackage{amsmath,amssymb}
\usepackage{mathtools}
\usepackage{stmaryrd}
\usepackage[all]{xy}
\usepackage[numbers,sort&compress]{natbib}
\usepackage{bm}

\usepackage[%
dvipdfm,%
colorlinks=true,%
bookmarks=true,%
bookmarksnumbered=true,%
pdftitle={Fourier-Laplace transform and isomonodromic deformations},%
pdfauthor={Daisuke Yamakawa}%
]{hyperref}

\newtheorem{theorem}{Theorem}[section]
\newtheorem{corollary}[theorem]{Corollary}
\newtheorem{proposition}[theorem]{Proposition}
\newtheorem{lemma}[theorem]{Lemma}
\newtheorem{conjecture}[theorem]{Conjecture}

\theoremstyle{definition}    
\newtheorem{definition}[theorem]{Definition}
\newtheorem{example}[theorem]{Example}
\newtheorem{remark}[theorem]{Remark}

\theoremstyle{remark}

\newenvironment{thm}{\begin{theorem}}{\end{theorem}}
\newenvironment{prp}{\begin{proposition}}{\end{proposition}}
\newenvironment{crl}{\begin{corollary}}{\end{corollary}}
\newenvironment{dfn}{\begin{definition}}{\end{definition}}
\newenvironment{lmm}{\begin{lemma}}{\end{lemma}}
\newenvironment{rmk}{\begin{remark}}{\end{remark}}

\newcommand{\thmref}[1]{Theorem~\ref{#1}}
\newcommand{\secref}[1]{Section~\ref{#1}}
\newcommand{\lmmref}[1]{Lemma~\ref{#1}}
\newcommand{\prpref}[1]{Proposition~\ref{#1}}
\newcommand{\crlref}[1]{Corollary~\ref{#1}}

\newcommand{\rmkref}[1]{Remark~\ref{#1}}

\newcommand{\Z}{\mathbb{Z}} 
\newcommand{\C}{\mathbb{C}} 
\newcommand{\bbP}{\mathbb{P}} 
\newcommand{\bbA}{\mathbb{A}} 

\newcommand{\GL}{\operatorname{GL}} 
\newcommand{\gl}{\operatorname{\mathfrak{gl}}}
\newcommand{\frg}{\mathfrak{g}} 
\newcommand{\frh}{\mathfrak{h}} 
\newcommand{\frt}{\mathfrak{t}} 
\newcommand{\ad}{\operatorname{ad}}

\newcommand{\End}{\operatorname{End}}
\newcommand{\Aut}{\operatorname{Aut}}
\newcommand{\Hom}{\operatorname{Hom}}
\newcommand{\Ker}{\operatorname{Ker}}
\newcommand{\range}{\operatorname{Im}}
\newcommand{\Coker}{\operatorname{Coker}}

\newcommand{\rank}{\operatorname{rank}}
\newcommand{\tr}{\operatorname{tr}}
\newcommand{\Id}{\mathrm{Id}}

\newcommand{\calO}{\mathcal{O}} 
\newcommand{\calK}{\mathcal{K}} 

\newcommand{\fp}[1]{\ensuremath{[\![#1]\!]} }

\newcommand{\res}{\operatorname*{res}}
\newcommand{\Res}{\operatorname*{Res}}

\newcommand{\calD}{\mathcal{D}} 
\newcommand{\frF}{\mathfrak{F}} 
\newcommand{\calF}{\mathcal{F}} 
\newcommand{\PWA}[1]{\C[#1]\langle \partial_{#1} \rangle} 

\newcommand{\lset}[2]%
{\left\{ \, \left. #1 \hspace{0.25em} \right| \, #2 \, \right\} }
\newcommand{\rset}[2]%
{\left\{ \, #1 \, \left| \hspace{0.25em} #2\right. \, \right\} } 
\newcommand{\ov}{\overline}
\newcommand{\wh}{\widehat}
\newcommand{\wt}{\widetilde}

\newcommand{\bbM}{\mathbb{M}}

\newcommand{\calS}{\mathcal{S}}
\newcommand{\calH}{\mathcal{H}}
\newcommand{\calC}{\mathcal{C}}
\newcommand{\Hol}{\operatorname{\mathcal{H}\hspace{-.1em}\mathit{ol}}}
\newcommand{\add}{\operatorname{\it add}}
\newcommand{\mc}{\operatorname{\it mc}}
\newcommand{\HD}{\operatorname{HD}}
\newcommand{\calV}{\mathcal{V}}
\newcommand{\calA}{\mathcal{A}}
\newcommand{\bfL}{\mathbf{L}}
\newcommand{\calR}{\mathcal{R}}

\title
{Fourier-Laplace transform and isomonodromic deformations
}

\author{Daisuke Yamakawa}
\address{Department of Mathematics,
Tokyo Institute of Technology,
Tokyo 152-8551, Japan}
\email{yamakawa@math.titech.ac.jp}
\thanks{This work was supported by 
JSPS Grant-in-Aid for Young Scientists (B) 
Grant Number 24740104.}


\begin{document}
\mathtoolsset{showonlyrefs=true}

\begin{abstract}
Using the Fourier-Laplace transform,
we describe the isomonodromy equations 
for meromorphic connections on the Riemann sphere 
with unramified irregular singularities 
as those for connections with 
a (possibly ramified) irregular singularity 
and a regular singularity.
This generalizes some results of Harnad and Woodhouse.
\end{abstract}

\maketitle

\tableofcontents

\section{Introduction}

Let
\[
\xymatrix{
V \ar@<1ex>[r]^P \ar@(lu,ld)[]_S & W \ar@<1ex>[l]^Q \ar@(ru,rd)[]^T
}
\]
be a diagram of finite-dimensional $\C$-vector spaces and linear maps.
Harnad~\cite{Har94} 
associated to such a diagram two meromorphic connections
\begin{equation}\label{eq:HD}
d - \left( S + Q(x 1_W -T)^{-1}P \right) dx, \quad 
d + \left( T + P(y 1_V -S)^{-1}Q \right) dy
\end{equation}
over the Riemann sphere $\bbP^1$, 
and observed that 
if $S, T$ are both regular semisimple,
then the isomonodromy equations for them~\cite{JMMS} coincide.

Harnad's duality of isomonodromic deformations%
\footnote{In this paper we use the term ``isomonodromic deformation'' 
in the de Rham sense, i.e., 
as a deformation of a meromorphic connection on $\bbP^1$
induced from some flat meromorphic connection on 
the product of $\bbP^1$ and the space of deformation parameters.
For the Betti approach to the isomonodromy in terms of monodromy/Stokes data, 
see \cite{JMU,Boa01,Boa13}.}
was generalized by Woodhouse~\cite{Woo07}. 
He examined the isomonodromy equation for a meromorphic connection 
$\nabla = d - A$ on a trivial vector bundle over $\bbP^1$
such that the one-form $A$ is holomorphic at infinity 
and the most singular coefficient of its Laurent expansion at each pole  
has distinct nonzero eigenvalues with no two differing by one.
He constructed some larger connection of the form 
$d - (x-T)^{-1}R\,dx$ with $T, R$ constant matrices
(such a connection is called a {\em generalized Okubo system}~\cite{Kaw})
whose quotient by $\Ker R$ is isomorphic to the original connection $\nabla$,
and then described the isomonodromy equation for $\nabla$ 
as that for the connection $d + (T + R/y)\,dy$ which relates to 
$d - (x-T)^{-1}R\,dx$ via \eqref{eq:HD}.
If $A$ has only at most logarithmic singularities, then 
$T$ is semisimple and his duality essentially reduces to Harnad's with $S=0$.
See also \cite{Boa12} for a generalization of Harnad's duality in another direction.

Correspondence~\eqref{eq:HD} is also used to construct 
the ``additive analogue'' of Katz's {\em middle convolution}~\cite{Kat}.
The middle convolution, 
which plays a key role in the study of rigid local systems~[loc.~cit],
is an operator with one parameter acting on 
the local systems on a punctured $\bbP^1$, 
and its analogue~\cite{DR00} acts on the Fuchsian systems 
(logarithmic connections on the trivial vector bundles on $\bbP^1$).
The two operators almost match up via the Riemann-Hilbert correspondence~\cite{DR07}
and both can be generalized to the irregular singular case: 
the generalized middle convolution~\cite{Ari} 
acts on the meromorphic connections on $\bbP^1$ and 
its additive analogue~\cite{Tak,Yam11} acts on the meromorphic connections 
on trivial bundles on $\bbP^1$. 
The counterpart of correspondence~\eqref{eq:HD} 
in the definition of middle convolution 
is the {\em Fourier-Laplace transform};
their direct relationship was found by Sanguinetti-Woodhouse~\cite{SW}.

The additive middle convolution is useful in its own right; 
Hiroe used it to construct Weyl group symmetries 
of the moduli spaces of meromorphic connections on trivial bundles on $\bbP^1$ 
with some local data around singularities fixed 
and applied it to solve the {\em additive irregular Deligne-Simpson problem},
i.e., gave a necessary and sufficient condition for 
the emptiness of such a moduli space~\cite{Hir}; 
it generalizes the result of Crawley-Boevey~\cite{CB} for the Fuchsian case,
the result of Boalch~\cite{Boa12} 
for the case where one pole is allowed to have order at most $3$,
and the result of himself and the author~\cite{HY} 
for the case where one pole is allowed to have arbitrary order.

It is natural to expect that Hiroe's symmetries induce symmetries of 
isomonodromic deformations.
Woodhouse's result does not imply it because 
the action of many generators of Weyl groups involves 
additive middle convolutions for connections 
which do not satisfy Woodhouse's assumption.


\medskip

In this paper we relax Woodhouse's assumption.
Let $\nabla = d-A$ be a meromorphic connection on $\calO_{\bbP^1}^{\oplus n}$ 
with a pole at infinity of order at most two.
Using our earlier results~\cite{Yam11}, 
we can then ``canonically'' express the one-form $A$ 
in the form $\left( S+Q(x-T)^{-1}P \right)dx$
and define the ``Harnad dual'' $d + \left( T + P(y-S)Q \right)\,dy$ to $\nabla$.
Assume further that $A$ is at most logarithmic at infinity (i.e., $S=0$) 
and that at each pole, 
in terms of a local coordinate $z$ vanishing there,
$\nabla$ is equivalent under the gauge action of $\GL_n(\C\fp{z})$ 
to an ``unramified normal form with non-resonant residue'', 
i.e., to a connection of the form $d - d \Lambda - L\,dz/z$,
where $\Lambda(z)$ is a diagonal matrix with entries in $z^{-1}\C[z^{-1}]$
and $L$ is an element of the Lie algebra  
$\frh = \{\, X \in \gl_n(\C) \mid [X,\Lambda]=0 \,\}$
such that $\ad_L \in \End(\frh)$ has no nonzero integral eigenvalue.%
\footnote{The normal forms are a basic notion in 
the formal classification theory of meromorphic connections;
see e.g.~\cite{BV}, where they are called the ``canonical forms''.
It is well-known that at each pole, 
if the most singular coefficient of the Laurent expansion of $A$ 
has distinct eigenvalues as in Woodhouse's case 
then $\nabla$ is equivalent (in the above sense) 
to an unramified normal form with non-resonant residue.}
We introduce the {\em admissible families} of such connections 
and (as its particular class) the isomonodromic deformations,
which may be viewed as a de Rham counterpart of the Poisson local systems 
established in \cite{Boa13} 
and generalizes the isomonodromic deformations of 
Jimbo, Miwa and Ueno~\cite{JMU}.
We show that if an admissible family is isomonodromic 
and the associated family on the Harnad dual side has a constant bundle rank, 
then it is also isomonodromic,
and vice versa provided that $A$ is irreducible or $\res_\infty A$ is invertible
(\thmref{thm:main}).
This result implies that the additive middle convolution preserves 
the isomonodromy property of admissible families of meromorphic connections 
under some assumptions (\crlref{crl:mc} and \rmkref{rmk:admissible}), 
generalizing a result of Haraoka-Filipuk~\cite{HF}.

\medskip

The organization of this paper is as follows.
In \secref{sec:FL}, 
we review the definition of the Harnad dual operation 
and its basic properties, especially the relationship 
with the Fourier-Laplace transform. 
The results presented in 
Sections~\ref{subsec:AHHP}--\ref{subsec:properties}
are not new, and Sections~\ref{subsec:minimal},~\ref{subsec:normal} 
contain a relationship between our canonical expression of 
connections in the form $\left( S+Q(x-T)^{-1}P \right)dx$ 
and the minimal extension of $\calD$-modules (see~\thmref{thm:minimal}).
\secref{sec:IMD} is devoted to show our main results.

\section{Fourier-Laplace transform and Harnad dual}\label{sec:FL}

Let us first recall what is the Fourier-Laplace transform.

Fix a base point $\infty \in \bbP^1$ and 
a standard coordinate $x$ on 
$\bbA^1 = \bbP^1 \setminus \{ \infty \}$.
Let $\calV$ be an algebraic vector bundle 
on some Zariski open subset $U$ of $\bbP^1$
equipped with a connection $\nabla_\calV$.
We regard it as a $\calD_U$-module in the obvious way.
Shrinking $U$ so that $U \subset \bbA^1$ if necessary,
let $j \colon U \to \bbA^1$ be the inclusion map 
and $j_{!*} \calV$ the minimal extension of $\calV$ (see e.g.~\cite{HTT,Kat90}), 
which is a $\calD_{\bbA^1}$-module and hence 
may be regarded as a module over the one-variable Weyl algebra $\PWA{x}$
(by taking the global sections $\Gamma$).
Let $\PWA{y}$ act on $\Gamma(j_{!*}\calV)$ by 
$y = -\partial_x$, $\partial_y =x$.
Then we obtain a new $\calD_{\bbA^1}$-module $\frF (j_{!*}\calV)$,
called the {\em Fourier transform} of $j_{!*}\calV$.
Since it is holonomic, 
we can take a maximal Zariski open subset $U' \subset \bbA^1$ such that 
$\calF(\calV) := \frF(j_{!*}\calV)|_{U'}$ is an algebraic vector bundle with connection,
called the {\em Fourier-Laplace transform} of $\calV$.

In this section we describe $\calF(\calV)$ 
when $\calV$ comes from a trivial holomorphic vector bundle over $\bbP^1$ 
equipped with a meromorphic connection $d-A$ 
which has a pole at $\infty$ of order at most two
and satisfy some nice condition at each pole.

\subsection{AHHP representation and Fourier transform}\label{subsec:AHHP}

\begin{lmm}[{\cite[Lemma 4]{Yam11}}]\label{lmm:exist}
Let $V$ be a nonzero finite-dimensional $\C$-vector space 
and $A$ an $\End_\C(V)$-valued meromorphic one-form with pole at $\infty$
of order at most two. 
Then there exists a finite-dimensional $\C$-vector space $W$ 
and an endomorphism 
\[
\gamma = 
\begin{pmatrix}
S & Q \\
P & T
\end{pmatrix}
\in \End_\C(V \oplus W)
\]
such that $A = \left( S + Q(x 1_W - T)^{-1}P \right) dx$.
\end{lmm}

Such an expression of meromorphic one-forms appears in \cite{AHP88,AHH90},
so we call it an {\em AHHP representation}.
We will explain how an AHHP representation relates to the Fourier transform
following Sanguinetti-Woodhouse~\cite{SW} (in a different convention).
Put $A(x) = \langle A, \partial_x \rangle$ (so $A=A(x)dx$)
and let $U \subset \bbA^1$ be the set of all non-singular points of $A$ in $\bbA^1$.
Define an injective left $\calD_U$-endomorphism 
$\varphi_A$ of $\calD_U \otimes_\C V$ by 
\[
\varphi_A (f v) = f (\partial_x + A(x) ) v  \quad ( f \in \calD_U,\ v \in V).
\]
Then we have the following short exact sequence:
\[
\xymatrix{
0 \ar[r] & 
\calD_U \otimes_\C V \ar[r]^{\varphi_A} &  
\calD_U \otimes_\C V \ar[r] & \calV \ar[r] & 0,
}
\]
where $\calV$ is the vector bundle $\calO_U \otimes_\C V$ 
equipped with the connection $d-A$, regarded as a left $\calD_U$-module,
and $\calD_U \otimes_\C V \to \calV$ is the map 
canonically induced from 
the $\calD_U$-module structure of $\calV$.
On the other hand, 
for a finite-dimensional $\C$-vector space $W$ and an endomorphism 
$\gamma \in \End_\C(V \oplus W)$,
define an injective left $\calD_{\bbA^1}$-module endomorphism $\varphi_\gamma$ of 
$\calD_{\bbA^1} \otimes_\C (V \oplus W)$ by 
\[
\varphi_\gamma \colon f  
\begin{pmatrix} v \\ w \end{pmatrix}
\mapsto 
f 
\begin{pmatrix} \partial_x + S & Q \\ -P & x - T \end{pmatrix}
\begin{pmatrix} v \\ w \end{pmatrix}
\quad ( f \in \calD_{\bbA^1},\ v \in V,\ w \in W),
\] 
where $S, Q, P, T$ are the blocks of $\gamma$,
and set $\calV_\gamma = \Coker \varphi_\gamma$.
The equality 
\[
\begin{pmatrix} \partial_x + S & Q \\ -P & x - T \end{pmatrix}
\begin{pmatrix} 1_V \\ (x -T)^{-1}P \end{pmatrix}
=
\begin{pmatrix} \partial_x + S + Q(x -T)^{-1}P \\ 0 \end{pmatrix}
\]
shows that
if $A(x) = S + Q(x 1_W - T)^{-1}P$, 
then the following diagram with exact rows commutes:
\[
\xymatrix{
0 \ar[r] &
\calD_U \otimes_\C V \ar[r]^{\varphi_A} \ar[d]^{\iota_1} &  
\calD_U \otimes_\C V \ar[r] \ar[d]^{\iota_2} & \calV \ar[r] & 0 \\
0 \ar[r] &
\calD_U \otimes_\C (V \oplus W) \ar[r]^{\varphi_\gamma} &  
\calD_U \otimes_\C (V \oplus W) \ar[r] & \calV_\gamma |_U \ar[r] & 0
}
\]
where $\iota_1, \iota_2$ are defined by 
\[
\iota_1 (f v) =   
\begin{pmatrix} f v \\ f(x 1_W -T)^{-1}Pv \end{pmatrix},
\quad
\iota_2 (f v) =   
\begin{pmatrix} f v \\ 0 \end{pmatrix}
\quad ( f \in \calD_U,\ v \in V).
\]
The commutativity and exactness imply that $\iota_2$ descends to  
a homomorphism $\iota \colon \calV \to \calV_\gamma |_U$.
Since the natural inclusion map 
$\calD_U \otimes_\C W \to \calD_U \otimes_\C (V \oplus W)$ 
enables us to identify each $\Coker \iota_j$ with 
$\calD_U \otimes_\C W$ and the homomorphism 
\[
\calD_U \otimes_\C W = \Coker \iota_1 \to \Coker \iota_2 = \calD_U \otimes_\C W
\]
induced from $\varphi_\gamma$, which is given by 
\[
f w \mapsto f(x-T)w \quad (f \in \calD_U,\ w \in W),
\]
is an isomorphism,
we see (e.g.\ from the snake lemma) that  
$\iota$ is an isomorphism.
Therefore a pair $(W,\gamma)$ 
as in \lmmref{lmm:exist}
give a $\calD_{\bbA^1}$-module $\calV_\gamma$ 
which is an extension of $\calV$ to $\bbA^1$.

Also, the Fourier transform $\frF(\calV_\gamma)$ of $\calV_\gamma$ 
is the cokernel of the endomorphism $\psi_\gamma$ of 
$\calD_{\bbA^1} \otimes_\C (V \oplus W)$ given by 
\[
\psi_\gamma \colon f  
\begin{pmatrix} v \\ w \end{pmatrix}
\mapsto 
f 
\begin{pmatrix} -y + S & Q \\ -P & \partial_y - T \end{pmatrix}
\begin{pmatrix} v \\ w \end{pmatrix}
\quad ( f \in \calD_U,\ v \in V,\ w \in W).
\] 
A similar argument based on the equality
\[
\begin{pmatrix} -y + S & Q \\ -P & \partial_y - T \end{pmatrix}
\begin{pmatrix} (y-S)^{-1}Q \\ 1_W \end{pmatrix}
=
\begin{pmatrix} 0 \\ \partial_y -T - P(y -S)^{-1}Q \end{pmatrix}
\]
shows that if we define $U'$ to be $\bbA^1$ minus the spectra of $S$,
then $\frF(\calV_\gamma) |_{U'}$ is isomorphic to 
the algebraic vector bundle $\calO_{U'} \otimes_\C W$ 
equipped with the connection $d + \left( T + P(y 1_V -S)^{-1}Q \right) dy$.

\subsection{Categorical treatment}\label{subsec:cat}

The categorical treatment of the previous arguments 
will make the story clearer.

Let $\calS$ be the category of pairs consisting of 
a holomorphically trivial vector bundle $\calV$ on $\bbP^1$ 
and a meromorphic connection $\nabla_\calV$ on $\calV$
having a pole at $\infty$ of order at most two.
The morphisms in $\calS$ are holomorphic bundle maps intertwining the connections.
We identify $\calS$ with the category of pairs $(V,A)$ 
consisting of a finite-dimensional $\C$-vector space $V$ 
and an $\End_\C(V)$-valued rational one-form $A$ 
having a pole at $\infty$ of order at most two.
The morphisms $(V,A) \to (V',A')$ in $\calS$
are then the linear maps $\varphi \colon V \to V'$ satisfying 
$A' \varphi = \varphi A$.

Let $\calH$ be the category of tuples $(V,W;\gamma)=(V,W;S,T,Q,P)$ 
consisting of two finite-dimensional $\C$-vector spaces $V,W$ and 
an endomorphism
\[
\gamma = \begin{pmatrix} S & Q \\ P & T \end{pmatrix}
\in \End_\C(V \oplus W).
\]
The morphisms $(V,W;\gamma) \to (V',W';\gamma')$ in $\calH$
are the pairs $(\varphi,\psi)$ of linear maps 
$\varphi \colon V \to V'$, $\psi \colon W \to W'$
satisfying $(\varphi \oplus \psi)\gamma = \gamma'(\varphi \oplus \psi)$.

The previous arguments lead to the definition of the following functor $\calH \to \calS$:
For $(V,W;\gamma)=(V,W;S,T,Q,P) \in \calH$, 
define an object $\Phi(V,W;\gamma)=(V,A)$ of $\calS$ by
\[
A = d - \left( S + Q(x 1_W -T)^{-1}P \right) dx.
\]
If $(\varphi,\psi)$ is a morphism from $(V,W;\gamma)$ to $(V',W';\gamma')$ in $\calH$,
then it is easy to see that 
$\varphi \colon V \to V'$ 
is a morphism from $\Phi(V,W;\gamma)$ to $\Phi(V',W';\gamma')$.
Thus we obtain a functor 
\[
\Phi \colon (V,W;\gamma) \mapsto \left( V, \left( S + Q(x 1_W -T)^{-1}P\right) dx \right),
\quad (\varphi,\psi) \mapsto \varphi
\]
from $\calH$ to $\calS$.
(We will denote it by $\Phi_x$ when emphasizing the choice of coordinate $x$.)
\lmmref{lmm:exist} implies that $\Phi$ is essentially surjective.

To treat vector bundles with connection 
on various Zariski open subsets in $\bbA^1$ at once,
we introduce the category $\calC$
of vector bundles with connection over the generic point of $\bbA^1$, i.e., 
the category of finite-dimensional $\C(x)$-vector spaces $\calV$ 
equipped with a $\C$-linear map 
$\nabla \colon \calV \to \calV \otimes_\C \C(x)\,dx$
satisfying the Leibniz rule. 
Let $\Hol(\PWA{x})$ be the category of holonomic (left) $\PWA{x}$-modules.
Then the restriction gives the functors 
\begin{alignat*}{2}
p &\colon \calS \to \calC; &\quad
(V, A) &\mapsto (\C(x) \otimes_\C V, d-A), \\
q &\colon \Hol(\PWA{x}) \to \calC; &\quad
M &\mapsto \C(x) \otimes_\C \Gamma(M),
\end{alignat*}
and the minimal extension gives a functor 
$e \colon \calC \to \Hol(\PWA{x})$.
It is natural to regard the Fourier-Laplace transform $\calF$ as the composite
\[
q \circ \frF \circ e \colon \calC \to \calC.
\]

The previous arguments show that 
$p \circ \Phi$ factors through the functor $\wt{\Phi}$ from $\calH$ to $\Hol(\PWA{x})$
given by $(V,W;\gamma) \to \calV_\gamma$: $p \circ \Phi = q \circ \wt{\Phi}$.

Also we have the functor $\sigma \colon \calH \to \calH$ defined by
\[
(V,W;S,T,Q,P) \mapsto (W,V;-T,S,P,-Q), \quad (\varphi,\psi) \mapsto (\psi,\varphi).
\]
The composite $\Phi_y \circ \sigma$ is described as 
\[
(V,W;\gamma) \mapsto \left( W,- \left( T + P(y 1_V -S)^{-1}Q \right) dy \right),
\quad (\varphi,\psi) \mapsto \psi,
\]
and the previous arguments show that 
\[
\wt{\Phi}_y \circ \sigma = \frF \circ \wt{\Phi}_x, 
\]
where the subscripts mean the choice of coordinate.

\subsection{Canonical section and Harnad dual}\label{subsec:HD}

Note that for fixed $(V,A) \in \calS$, 
an object $(V,W;\gamma) \in \calH$ satisfying $\Phi(V,W;\gamma)=(V,A)$ 
is not unique.
However, we can show that a stable object in the following sense is 
essentially unique:

\begin{dfn}\label{dfn:stable}
An object $(V,W;\gamma) \in \calH$ 
is said to be {\em stable} if 
the following two conditions hold:
\begin{enumerate}
\item if a subspace $W' \subset W$ satisfies 
$\gamma (V \oplus W') \subset V \oplus W'$, then $W'=W$;
\item if a subspace $W' \subset W$ satisfies 
$\gamma (0 \oplus W') \subset 0 \oplus W'$, then $W'=0$.
\end{enumerate}
\end{dfn}

\begin{prp}[{\cite[Theorem 1]{Yam11}}]\label{prp:unique}
For any $(V,A) \in \calS$ with $V \neq 0$, 
there exists a stable object $(V,W;\gamma) \in \calH$ satisfying 
$\Phi(V,W;\gamma)=(V,A)$.
If another stable object $(V,W';\gamma')$ satisfies the same condition,
then there exists an isomorphism $f \colon W \xrightarrow{\simeq} W'$
such that $\gamma' (1_V \oplus f) = (1_V \oplus f) \gamma$;
in particular, the two objects are isomorphic.
\end{prp}

In fact, we can construct a ``section'' $\kappa \colon \calS \to \calH$ 
of $\Phi$ such that $\kappa(V,A)$ is stable for any $(V,A) \in \calS$ as follows:

Let $(V,A) \in \calS$.
Label the poles of $A$ in $\bbA^1$ as 
$t_1, t_2, \dots , t_m$ and write  
\[
A(x) = A_0 + \sum_{i=1}^m \sum_{j=1}^{k_i}
\frac{A^{(i)}_j}{(x-t_i)^j},
\quad S,\, A^{(i)}_j \in \End_\C(V),
\] 
where $k_i \in \Z_{>0}$ is the pole order of $A(x)$ at $x=t_i$.
For $i=1,2, \dots , m$, put $x_i =x-t_i$ and 
let $A_i = \sum_{j=1}^{k_i} A^{(i)}_j x_i^{-j} dx$
be the principal part of the Laurent expansion of $A$ at $x_i=0$.
We set
\[
\wh{A}_i = x_i^{k_i} \langle A_i , \partial_{x_i} \rangle 
= \sum_{j=1}^{k_i} A^{(i)}_j x_i^{k_i- j},
\]
which we regard as an element of  
\[
\End_\C(V) \otimes_\C \calR_i 
\simeq \End_{\calR_i}(V \otimes_\C \calR_i), 
\quad \calR_i := \C[x_i]/(x_i^{k_i}).
\]
Also set
\[
W_i = V \otimes_\C \calR_i/\Ker \wh{A}_i.
\]
Note that we have a natural isomorphism
\begin{equation}\label{eq:species1}
\Hom_\C(V,W_i) \simeq \Hom_{\calR_i}(V \otimes_\C \calR_i,W_i);\quad
Y \mapsto \left[ \wt{Y} \colon v \otimes a \mapsto Y(v)a \right],
\end{equation}
and that the tensor-hom adjunction and 
the non-degenerate pairing 
\[
\calR_i \otimes_\C \calR_i \to \C; 
\quad f(x_i) \otimes g(x_i) \mapsto 
\res_{x_i=0} \left( x_i^{-k_i}f(x_i)g(x_i) \right)
\]
yield
\begin{equation}\label{eq:species2}
\begin{aligned}
\Hom_\C(W_i,V) &\simeq 
\Hom_\C(W_i \otimes_{\calR_i} \calR_i, V) \\
&\simeq \Hom_{\calR_i}(W_i,\Hom_\C(\calR_i,V)) \\
&\simeq \Hom_{\calR_i}(W_i,V \otimes_\C \calR_i^*) 
\simeq \Hom_{\calR_i}(W_i,V \otimes_\C \calR_i),
\end{aligned}
\end{equation}
under which a linear map
$X \colon W_i \to V$
corresponds to the $\calR_i$-homomorphism
\[
\wt{X} : W_i \to V \otimes_\C \calR_i;\quad 
w \mapsto \sum_{j=1}^{k_i} X(x_i^{j-1} w) x_i^{k_i-j}.
\]
Now decompose $\wh{A}_i$ as $\wh{A}_i = \wt{Q}_i \wt{P}_i$,
where 
$\wt{P}_i \colon V \otimes_\C \calR_i \to W_i$
is the natural projection and
$\wt{Q}_i \colon W_i \to V \otimes_\C \calR_i$ is 
the injective homomorphism induced from $\wh{A}_i$.
These induce linear maps  
$P_i \colon V \to W_i$,
$Q_i \colon W_i \to V$ 
through \eqref{eq:species1}, \eqref{eq:species2}, respectively.  
Let $N_i \in \End_{\calR_i}(W_i)$ be
the endomorphism representing the multiplication by $x_i$.
Then a direct calculation shows
\[
\wh{A}_i = \wt{Q}_i \wt{P}_i 
= \sum_{j=1}^l Q_i N_i^{j-1} P_i x_i^{k_i-j} 
= x_i^{k_i} Q_i (x_i\,1_{W_i} -N_i)^{-1} P_i,
\]
and hence
\[
A_i = Q_i (x_i\,1_{W_i} -N_i)^{-1} P_i\, dx_i.
\]
Define $\kappa(V,A)=(V,W;\gamma)$ by $W = \bigoplus_{i=1}^m W_i$ and  
\begin{gather*}
S=A_0 \in \End_\C(V), \quad
T=\bigoplus_{i=1}^m (t_i\,1_{W_i} + N_i) \in \End_\C(W), \\
Q = 
\begin{pmatrix}
Q_1 & \cdots & Q_m
\end{pmatrix}
\in \Hom_\C(W,V),
\quad
P = 
\begin{pmatrix}
P_1 \\ \vdots \\ P_m
\end{pmatrix}
\in \Hom_\C(V,W).
\end{gather*}
Then 
\[
S+Q(x 1_W -T)^{-1}P 
= A_0 + \sum_{i=1}^m Q_i \left( (x-t_i) 1_{W_i} -N_i \right)^{-1} P_i 
= A(x).
\]
Hence $\kappa(V,A)$ is an object of $\calH$ satisfying 
$\Phi(\kappa(V,A))=(V,A)$.

Any morphism $\varphi \colon (V,A) \to (V',A')$ in $\calS$ induces 
a morphism $\psi$ from $\kappa(V,A)$ to $\kappa(V',A')=(V',W';\gamma')$ 
as follows:
Take a subset $I \subset \{ 1,2, \dots , m \}$ so that 
$\{\, t_i \mid i \in I \,\}$ is the set of common poles of $A, A'$. 
For $i \in I$, 
let $k'_i$ be the pole order of $A'$ at $t_i$,
put $l_i = \max\{ k_i,k'_i \}$ and set
\[
\calR'_i = \C[x_i]/(x_i^{k'_i}), 
\quad \calR''_i = \C[x_i]/(x_i^{l_i}).
\]
Recall that  
$A'$ induces $\wh{A}'_i \in \End_{\calR'_i}(V' \otimes_\C \calR'_i)$ 
and $W'_i = V' \otimes_\C \calR'_i/\Ker \wh{A}'_i$. 
The map $\varphi$ induces a $\calR''_i$-homomorphism 
\[
\varphi \otimes 1 \colon V \otimes_\C \calR''_i/\Ker x_i^{l_i-k_i} \wh{A}_i \to
V' \otimes_\C \calR''_i/\Ker x_i^{l_i-k'_i} \wh{A}'_i,
\]
and the natural projections 
$\C[x_i]/(x_i^{l+k}) \to \left( \C[x_i]/(x_i^{l+k}) \right)/(x_i^k) 
= \C[x_i]/(x_i^l)$ for $k,l \geq 0$ give isomorphisms 
\begin{align*}
V \otimes_\C \calR''_i/\Ker x_i^{l_i-k_i} \wh{A}_i &\simeq
V \otimes_\C \calR_i/\Ker \wh{A}_i =W_i, \\
V' \otimes_\C \calR''_i/\Ker x_i^{l_i-k'_i} \wh{A}'_i &\simeq
V' \otimes_\C \calR'_i/\Ker \wh{A}'_i = W'_i,
\end{align*}
which induce a map $\psi_i \colon W_i \to W'_i$ from $\varphi \otimes 1$ for $i \in I$.
Note that we have 
the natural projection $W \twoheadrightarrow \bigoplus_{i \in I} W_i$
and injection $\bigoplus_{i \in I} W'_i \hookrightarrow W'$.
Let $\psi \colon W \to W'$ be the composite 
\[
W \twoheadrightarrow \bigoplus_{i \in I} W_i 
\xrightarrow{\bigoplus_{i \in I} \psi_i} \bigoplus_{i \in I} W'_i
\hookrightarrow W'.
\]
Then one can easily check that $\kappa(\varphi):=(\varphi,\psi)$
is a morphism from $(V,W;\gamma)$ to $(V',W';\gamma')$ in $\calH$.

Thus we obtain a functor $\kappa \colon \calS \to \calH$ satisfying
$\Phi \circ \kappa = \Id$.
(We will denote it by $\kappa_x$ when emphasizing the choice of coordinate.)

\begin{dfn}\label{dfn:canonical}
We call the functor $\kappa$ the \emph{canonical section} of $\Phi$.
\end{dfn}

\begin{prp}\label{prp:canonical}
The object $(V,W;\gamma)=\kappa(V,A)$ is stable for any $(V,A) \in \calS$.
\end{prp}

\begin{proof}
If a subspace $W' \subset W$ satisfies 
$\gamma (V \oplus W') \subset V \oplus W'$, 
then in particular it is $T$-invariant and hence 
\[
W' = \bigoplus_{i=1}^m (W_i \cap W'), \quad N_i(W_i \cap W') \subset W_i \cap W'.
\]
Furthermore, the condition $\gamma (V \oplus W') \subset V \oplus W'$ 
implies $\range\wt{P}_i \subset W_i \cap W'$ for all $i$.
Since $\wt{P}_i$ are surjective,
we obtain $W_i \cap W' = W_i$ for all $i$, i.e., $W'=W$.

If $W'$ satisfies $\gamma (0 \oplus W') \subset 0 \oplus W'$,
then it is $T$-invariant and $\Ker \wt{Q}_i \supset W_i \cap W'$ for all $i$.
Since $\wt{Q}_i$ are injective, $W_i \cap W' =0$ for all $i$, i.e., $W'=0$.
\end{proof}

Using the canonical section, we introduce the following functor,
which may be regarded as an ``additive analogue'' of the Fourier-Laplace transform:

\begin{dfn}\label{dfn:HD}
We call $\HD := \Phi \circ \sigma \circ \kappa \colon \calS \to \calS$ 
the \emph{Harnad dual functor}
and $\ov{\HD} := \Phi \circ \sigma^{-1} \circ \kappa \colon \calS \to \calS$ 
the \emph{inverse Harnad dual functor}.
\end{dfn}

\subsection{Properties of canonical section and Harnad dual}\label{subsec:properties}

The canonical section $\kappa$ has some nice properties.
First, $\kappa$ preserves the natural direct sum operation
(the proof is immediate):
\begin{prp}\label{prp:sum}
For $(V,A), (V',A') \in \calS$, there exists a natural isomorphism
\[
\kappa (V \oplus V', A \oplus A') \simeq \kappa (V,A) \oplus \kappa (V',A')
\]
of the form $(1_{V \oplus V'}, \psi)$.
\end{prp}

Next, $\kappa$ preserves the irreducibility in the following sense:

\begin{dfn}\label{dfn:irreducible}
(i) An object $(V,A) \in \calS$ is \emph{irreducible}
if there exists no subspace $V' \subset V$ 
such that $A(V' \otimes_\C \C(x)) \subset V' \otimes_\C \C(x)\,dx$
except $V'=0,V$.

(ii) An object $(V,W;\gamma) \in \calH$ is said to be \emph{irreducible} 
if there exists no pair of subspaces $V' \subset V,\, W' \subset W$
such that $\gamma(V' \oplus W') \subset V' \oplus W'$
except $(V',W')=(0,0), (V,W)$.
\end{dfn}

\begin{prp}[{\cite[Lemmas 8, 9]{Yam11}}]\label{prp:irreducible}
{\rm (i)} Suppose that $(V,W;\gamma) \in \calH$ is irreducible.
If $V\neq 0$, then it is stable, while 
if $W\neq 0$, then $\sigma(V,W;\gamma)$ is stable.

{\rm (ii)} If $(V,W;\gamma) \in \calH$ and $V \neq 0$, 
then $(V,W;\gamma)$ is irreducible if and only if $\Phi(V,W;\gamma)$ 
is irreducible. In particular, 
an object $(V,A) \in \calS$ with $V \neq 0$ is irreducible 
if and only if $\kappa(V,A) \in \calH$ is irreducible.
\end{prp}

Using the above proposition we can show that 
the functor $\HD$ also preserves the irreducibility
and has a sort of inversion formula:

\begin{thm}[{\cite[Theorem 7]{Yam11}}]\label{thm:inversion}
Suppose that $(V,A) \in \calS$ is irreducible 
and not isomorphic to an object of the form $(\C,c\,dx),\, c \in \C$.
Then $\HD(V,A)$ is also irreducible and 
\[
\ov{\HD} \circ \HD (V,A) \simeq (V,A).
\]
\end{thm}

The functor $\Phi$ also has some important geometric properties. 
In the rest of this subsection, we fix 
two finite-dimensional $\C$-vector spaces $V \neq 0$, $W$
and endomorphisms $S \in \End_\C(V)$, $T \in \End_\C(W)$.
Set 
\[
\bbM \equiv \bbM(V,W) = \{\, (Q,P) \mid Q \in \Hom_\C(W,V),\ P \in \Hom_\C(V,W)\,\},
\]
which we equip with a symplectic form $\tr dQ \wedge dP$.
Let $G_S \subset \GL(V)$ (resp.\ $G_T \subset \GL(W)$) 
be the centralizer of $S$ (resp.\ $T$)
and $\frg_S$ (resp.\ $\frg_T$) its Lie algebra.
The group $G_S \times G_T$ acts on $\bbM$ by
\[
(g,h) \colon (Q,P) \mapsto (h Q g^{-1}, g P h^{-1}). 
\]
We label the eigenvalues of $T$ and their algebraic multiplicities as 
$t_i, k_i$, $i=1,2, \dots , m$ and set
\[
\wt{G}(T) = \prod_{i=1}^m \Aut_{\calR_i}(V \otimes_\C \calR_i),
\quad \calR_i = \C[x_i]/(x_i^{k_i}).
\]
Let $\wt{\frg}(T)$ be its Lie algebra:
\[
\wt{\frg}(T) = 
\bigoplus_{i=1}^m \gl(V) \otimes_\C \calR_i
\simeq \bigoplus_{i=1}^m \End_{\calR_i}(V \otimes_\C \calR_i),
\]
and set
\[
\wt{\frg}^*(T) = 
\bigoplus_{i=1}^m \bigoplus_{j=1}^{k_i} \gl(V) x_i^{-j}dx_i,
\] 
which we embed into $\gl(V) \otimes_\C \C(x)\,dx$ via $x_i = x-t_i$
and identify with the $\C$-dual to $\wt{\frg}(T)$ using the pairing
\[
(A,X) := \sum_{i=1}^m \Res_{x_i=0} \tr X_i A_i, \quad
X =(X_i) \in \wt{\frg}(T),\ A=(A_i) \in \wt{\frg}^*(T).
\]
We let $\wt{G}(T)$ act on $\bbM$ as follows:
For $i=1,2, \dots , m$, let $W_i \subset W$ be 
the generalized $t_i$-eigenspace for $T$ and 
\begin{itemize}
\item $N_i := T|_{W_i} - t_i\,1_{W_i} \in \End_\C(W_i)$;
\item $Q \mapsto Q_i$ the restriction $\Hom_\C(W,V) \to \Hom_\C(W_i,V)$;
\item $P \mapsto P_i$ the projection $\Hom_\C(V,W) \to \Hom_\C(V,W_i)$.
\end{itemize} 
For $g=(g_i) \in \wt{G}(T)$ and $Q \in \Hom_\C(W,V)$, 
define $g \cdot Q =Q' \in \Hom_\C(W,V)$ by
\[
\wt{Q}'_i = g_i \cdot \wt{Q}_i \in \Hom_{\calR_i}(W_i,V \otimes_\C \calR_i),
\]
or equivalently,
\begin{equation}\label{eq:action-Q}
Q'_i = \sum_{j=0}^{k_i-1} g^{(i)}_j Q_i N_i^j, \quad g_i(x_i) 
= \sum_{j=0}^{k_i-1}g^{(i)}_j x_i^j. 
\end{equation}
Similarly, for $P \in \Hom_\C(V,W)$ define $g \cdot P =P' \in \Hom_\C(V,W)$ by
\[
\wt{P}'_i = \wt{P}_i \cdot g_i^{-1} \in \Hom_{\calR_i}(V \otimes_\C \calR_i,W_i),
\]
or equivalently,
\begin{equation}\label{eq:action-P}
P'_i = \sum_{j=0}^{k_i-1} N_i^j P_i\, \bar{g}^{(i)}_j,
\quad g_i(x_i)^{-1} = \sum_{j=0}^{k_i-1}\bar{g}^{(i)}_j x_i^j
\end{equation}
Then $g \colon (Q,P) \mapsto (g \cdot Q, g \cdot P)$ gives 
an action of $\wt{G}(T)$ on $\bbM$ preserving the symplectic structure.
Note that if $(V,W;S,T,Q,P) \in \calH$ is stable, 
then $(V,W;S,T,g \cdot Q, g \cdot P)$ is also stable for any $g \in \wt{G}(T)$.
Let $\bbM^{st}$ be the set of all $(Q,P)$
such that $(V,W;S,T,Q,P)$ is stable.

\begin{prp}[\cite{AHP88,AHH90,Yam11}]\label{prp:AHHP}
{\rm (i)} The map 
\[
\Phi_T \colon \bbM \to \wt{\frg}^*(T);
\quad (Q,P) \mapsto Q(x 1_W -T)^{-1}P\, dx
\]
is a moment map generating the $\wt{G}(T)$-action.

{\rm (ii)} The action of $G_T$ on $\bbM^{st}$ is free and proper.

{\rm (iii)} The map $\Phi_T$ is $G_T$-invariant and induces a Poisson embedding
\[
\bbM^{st}/G_T \hookrightarrow \wt{\frg}^*(T),
\]
which induces a symplectomorphism from 
the symplectic quotient of $\bbM^{st}$ by the $G_T$-action 
along each $G_T$-coadjoint orbit onto a $\wt{G}(T)$-coadjoint orbit. 
\end{prp}

The following lemma will be used later:

\begin{lmm}\label{lmm:stability}
Let $(Q,P) \in \bbM^{st}$.

{\rm (i)} If $C \in \frg_T$ satisfies
\[
Q(x 1_W -T)^{-1}CP=0,
\]
then $C=0$.

{\rm (ii)} If $Q' \in \Hom_\C(W,V)$ and $P' \in \Hom_\C(V,W)$ satisfy
\[
Q'(x 1_W -T)^{-1} P = Q (x 1_W -T)^{-1} P',
\]
then there exists a unique $C \in \frg_T$ such that $Q'=QC$, $P'=CP$.
\end{lmm}

\begin{proof}
(i) As $C$ commutes with $T$, it has the form 
$C=\bigoplus_i C_i$ with $C_i \in \End_\C(W_i)$ and
\[
\wt{Q_i C_i} = \sum_{j \geq 1} Q_i C_i N_i^{j-1} x_i^{k_i-j} 
= \sum_{j \geq 1} Q_i N_i^{j-1} x_i^{k_i-j} C_i
= \wt{Q}_i C_i.
\]
The assumption implies $\wt{Q}_i C_i \wt{P}_i =0$. 
Since $\wt{Q}_i$ and $\wt{P}_i$ are injective and surjective, respectively,
we obtain $C_i =0$.

(ii) The assumption tells us that
$(Q',-P') \in \bbM(V,W) \simeq T_{(Q,P)}\bbM(V,W)$ 
is contained in $\Ker (d \Phi_T)_{(Q,P)}$. 
\prpref{prp:AHHP} implies 
\[
\Ker (d \Phi_T)_{(Q,P)} = T_{(Q,P)} \left( G_T \cdot (Q,P) \right),
\] 
which shows the assertion.
\end{proof}

\subsection{Stable objects and minimal extensions}\label{subsec:minimal}

The following proposition shows that a stable object gives 
the minimal extension under some assumption:

\begin{prp}\label{prp:minimal}
Let $(V,W;\gamma)$ be a stable object of $\calH$.
Label the eigenvalues of $T$ as $t_i$, $i=1,2, \dots , m$.
For each $i$, let $W_i \subset W$ be the generalized $t_i$-eigenspace for $T$
and 
\begin{itemize}
\item $N_i \in \End_\C(W_i)$ the nilpotent part of $T |_{W_i}$,
\item $Q_i \in \Hom_\C(W_i,V)$ the restriction of $Q$ to $W_i$,  
\item $P_i \in \Hom_\C(V,W_i)$ the projection of $P$ to $W_i$.
\end{itemize}
Assume that for each $i$ and $k \in \Z$, the map from $\Ker N_i$ to $\Coker N_i$ 
induced from $(P_i Q_i + k 1_{W_i}) |_{\Ker N_i}$ is an isomorphism.
Then the $\calD_{\bbA^1}$-module $\calV_\gamma = \wt{\Phi}(V,W;\gamma)$ satisfies 
$\calV_\gamma \simeq j_{!*}j^*\calV_\gamma$,
where $j$ is the inclusion map of 
$U := \bbA^1 \setminus \{ t_1, t_2, \dots ,t_m \}$ into $\bbA^1$.
\end{prp}

\begin{proof}
It is well-known (see e.g.~\cite[Lemma~2.9.1]{Kat90}) 
that there is an isomorphism $\calV_\gamma \simeq j_{!*}j^*\calV_\gamma$
which is an identity on $U$ if and only if
\[
\Hom_{\calD_{\bbA^1}}(\delta_{t_i},\calV_\gamma) = 0,
\quad
\Hom_{\calD_{\bbA^1}}(\calV_\gamma,\delta_{t_i}) = 0
\quad (i=1,2, \dots , m),
\] 
where $\delta_{t_i} := \calD_{\bbA^1}/\calD_{\bbA^1}(x-t_i)$.
Assume that there is a nonzero homomorphism 
$\delta_{t_i} \to \calV_\gamma$ for some $i$.
Taking the Fourier transform and restricting to 
the open subset $U'$ equal to $\bbA^1$ minus the spectra of $S$, 
we then obtain a nonzero homomorphism of connections
\[
(\calO_{U'},d+t_i\,dy) \to 
\left( \calO_{U'} \otimes_\C W, d + \left( T + P(y 1_V -S)^{-1}Q \right) dy \right).
\]
In particular, we see that there is a 
nonzero $W$-valued holomorphic function $w$ on $U'$
such that
\[
\partial_y w + \left( T + P(y 1_V -S)^{-1}Q \right) w = t_i w.
\]
Let $w(y)= \sum_{l=0}^\infty w_l y^{k-l}$, $w_0 \neq 0$
be the Laurent expansion of $w$ at $y=\infty$.
Using the expansion 
\[
P(y 1_V -S)^{-1}Q = \sum_{l \geq 0} PS^l Q y^{-l-1},
\]
we obtain
\[
(T - t_i\,1_W) w_0 =0, \quad (T - t_i\,1_W) w_1 + (PQ + k 1_W) w_0 = 0.
\]
Hence
\[
w_0 \in \Ker N_i \subset W_i, \quad  (P_i Q_i + k 1_{W_i})(w_0) \in \range N_i,
\] 
which contradicts the assumption.
Hence $\Hom_{\calD_{\bbA^1}}(\delta_{t_i},\calV_\gamma)=0$ for all $i$.
The dual argument also shows that 
$\Hom_{\calD_{\bbA^1}}(\calV_\gamma,\delta_{t_i})=0$ for all $i$. 
\end{proof}

\begin{rmk}
In the above proof the stability property of $(W,\gamma)$ is not used.
However, if $(W,\gamma)$ is not stable, then the map $\Ker N_i \to \Coker N_i$ 
induced from $P_i Q_i |_{\Ker N_i}$ is not an isomorphism for some $i$. 
Indeed, assume there is a nonzero subspace $W' \subset W$ such that
$\gamma(0 \oplus W') \subset 0 \oplus W'$.
Then $W'_i := W' \cap W_i \neq 0$ for some $i$ and it satisfies 
$N_i(W'_i) \subset W'_i$ and $Q_i(W'_i)=0$.
Since $N_i$ is nilpotent, $\Ker N_i \cap W'_i \neq 0$. 
Hence $P_i Q_i |_{\Ker N_i}$ is not injective.
Similarly, if there is a proper subspace $W' \subset W$ such that
$\gamma(V \oplus W') \subset V \oplus W'$, then 
the projection of $P_i Q_i$ to $\Coker N_i$ is not surjective.
\end{rmk}

\subsection{Normal forms}\label{subsec:normal}

Now we will give a local condition for $(V,A) \in \calS$ which is sufficient for 
that $(p \circ \HD)(V,A)$ is isomorphic to the Fourier-Laplace transform of $p(V,A)$.

For $t \in \bbP^1$, denote by $\calO_t$ the formal completion of 
the ring of germs at $t$ of holomorphic functions 
and by $\calK_t$ its field of fractions. 
Fix a maximal torus $\frt$ of $\gl(V)$.

\begin{dfn}\label{dfn:normal}
(i) For $t \in \bbP^1$, an element of $\frt(\calK_t)/\frt(\calO_t)$ 
is called an (unramified) {\em irregular type} at $t$.

(ii) Let $\Lambda$ be an irregular type at $t \in \bbP^1$.
Take a local coordinate $z$ vanishing at $t$ 
and regard $\Lambda$ as an element of 
$z^{-1}\frt[z^{-1}] \simeq \frt(\calK_t)/\frt(\calO_t)$.
Then for $L \in \gl(V)$ with $L \Lambda = \Lambda L$, 
the connection 
\[
d - d \Lambda - L\,dz/z
\] 
is called a {\em normal form} with irregular type $\Lambda$.
\end{dfn}

It is useful to calculate $\kappa_z(V,A^0)$ 
for a normal form $d-A^0$, $A^0 = d\Lambda + L\,dz/z$ at $t \in \bbP^1$.
Take a basis of $V$ so that $\frt$ is identified with the standard maximal torus,
and label the nonzero diagonal entries of $\Lambda$ as 
$\lambda_1, \lambda_2, \dots , \lambda_d$.
Set $\lambda_0 \equiv 0$ for convenience.
For $a=0,1, \dots , d$, set
\[
V_a = \Ker (\Lambda - \lambda_a 1_V \colon V \to V \otimes_\C \calK_t/\calO_t).
\]
Then we have direct sum decompositions
\[
V = \bigoplus_{a=0}^d V_a, \quad
\Lambda = \bigoplus_{a=0}^d \lambda_a\,1_{V_a},
\quad
L = \bigoplus_{a=0}^d L_a,
\]
where $L_a \in \End_\C(V_a)\ (a=0,1, \dots , d)$.
Thus we have a natural isomorphism
\[
\kappa_z(V,d\Lambda + L\,dz/z) \simeq 
\bigoplus_{a=0}^d \kappa_z(V_a, d\lambda_a + L_a\,dz/z)
\]
by \prpref{prp:sum}. For each $a \neq 0$, let us calculate 
\[
(V_a,W_a;0,N_a,X_a,Y_a) := \kappa_z(V_a, d\lambda_a + L_a\,dz/z).
\]
For $a \neq 0$, let $k_a$ be the pole order of $d\lambda_a$.
By definition, we have
\[
W_a = V_a \otimes_\C \left( \C[z]/(z^{k_a}) \right) / \Ker (d\lambda_a + L_a\,dz/z)^{\wedge}.
\]
Write $\partial_z \lambda_a = \sum_{j=2}^{k_a} \lambda_{a,j} z^{-j}$.
Since $\lambda_{a,k_a}$ is a nonzero scalar, 
$(d\lambda_a + L_a\,dz/z)^{\wedge}$ is invertible 
and hence $W_a = V_a \otimes_\C \C[z]/(z^{k_a})$.
By the definition, $N_a \colon W_a \to W_a$ is the multiplication by $z$, 
$Y_a \colon V_a \to W_a$ is the inclusion map,
and $X_a \colon W_a \to V_a$ is given by 
\[
\wt{X}_a = (z^{k_a} \partial_z \lambda_a) 1_{W_a} + L_a \otimes z^{k_a-1} 
\in \End_{\C[z]/(z^{k_a})}(W_a),
\]
i.e., for $v \in V_a$ and $l=0,1, \dots , k_a-1$,
\[
X_a (v \otimes z^l) = 
\begin{cases}
\lambda_{a,l+1} v & (l>0) \\
L_a v & (l=0).
\end{cases}
\]
Under the identification $W_a = V_a^{\oplus k_a}$ 
induced from the basis $z^{k_a-1}, z^{k_a -2}, \dots , 1$ 
of $\C[z]/(z^{k_a})$, 
the linear maps $X_a, Y_a, N_a$ are thus respectively expressed as
\begin{equation}\label{eq:matrix}
\begin{aligned}
X_a &= 
\begin{pmatrix} 
\lambda_{a,k_a} & \lambda_{a,k_a-1} & 
\cdots & \lambda_{a,2} & L_a
\end{pmatrix}, \\ 
Y_a &= 
\begin{pmatrix} 
0 \\ \vdots \\ 0 \\ 1_{V_a}
\end{pmatrix}, \quad
N_a = 
\begin{pmatrix}
0     & 1_{V_a}   &        & 0             \\
      & 0               & \ddots &               \\
      &                 & \ddots & 1_{V_a} \\
0     &                 &        & 0 
\end{pmatrix}.
\end{aligned}
\end{equation}
On the other hand, for $a=0$, 
the space $W_0$ is given by the quotient $V_0/\Ker L_0$
and $N_0=0$.
The map $Y_0 \colon V_0 \to W_0$ is the projection and 
$X_0 \colon W_0 \to V_0$ is the map induced from $L_0$.

Based on the above observation, we show the following theorem:

\begin{thm}\label{thm:minimal}
Let $(V,A) \in \calS$ and label its poles in $\bbA^1$ as 
$t_i$, $i=1,2, \dots , m$.
Assume that for any $i$, 
there exists $\wh{g}_i \in \Aut_{\C\fp{x_i}}(V \otimes_\C \C\fp{x_i})$ 
and a normal form $d-d\Lambda_i-L_i\,dx_i/x_i$ at $x=t_i$ such that 
\[
\wh{g}_i ^{-1} \circ (d - A) \circ \wh{g}_i = d-d\Lambda_i-L_i\,dx_i/x_i + F(x_i)\,dx_i
\]
for some $F \in \End_{\C\fp{x_i}}(V \otimes_\C \C\fp{x_i})$.
Assume further that for each $i$, the restriction $L^{(i)}_0$ of $L_i$ to 
the subspace 
\[
\Ker (\Lambda_i \colon V \to V \otimes_\C \calK_t/\calO_t) \subset V
\]
satisfies 
\[
\Ker \left( L^{(i)}_0 \left( L^{(i)}_0 + k \right) \right) = \Ker L^{(i)}_0 
\quad (k \in \Z).
\]
Then $(\wt{\Phi} \circ \kappa)(V,A) \simeq (e \circ p)(V,A)$.
\end{thm}

\begin{proof}
Fix $i$ for the moment. As above we label the nonzero diagonal entries of $\Lambda_i$ 
as $\lambda^{(i)}_a$, $a=1,2, \dots , d_i$
and set $\lambda^{(i)}_0 \equiv 0$.
We then have the associated direct sum decompositions 
\[
V = \bigoplus_{a=0}^{d_i} V^{(i)}_a,
\quad
\Lambda_i = \bigoplus_{a=0}^{d_i} \lambda^{(i)}_a\,1_{V^{(i)}_a},
\quad
L_i = \bigoplus_{a=0}^{d_i} L^{(i)}_a.
\]
For $a=0,1, \dots , d_i$, let 
\[
(V^{(i)}_a,W^{(i)}_a;0,N^{(i)}_a, X^{(i)}_a, Y^{(i)}_a)
= \kappa_{x_i}(V^{(i)}_a, d\lambda^{(i)}_a + L^{(i)}_a\,dx_i/x_i), 
\]
and also
\[
(V,W_i;0,N_i,X_i,Y_i) 
= \bigoplus_{a=0}^{d_i} (V^{(i)}_a,W^{(i)}_a;0,N^{(i)}_a, X^{(i)}_a, Y^{(i)}_a).
\]
Then \prpref{prp:sum} implies that 
$(V,W_i;0,N_i,X_i,Y_i)$ is stable and 
\[
X_i (x_i\,1_{W_i} - N_i)^{-1} Y_i\,dx_i 
= d\Lambda_i + L_i\,dx_i/x_i.
\]
By \eqref{eq:matrix}, for $a \neq 0$
we have isomorphisms 
\[
\Ker N^{(i)}_a \simeq V^{(i)}_a \simeq \Coker N^{(i)}_a
\]
in terms of which, for any $k \in \Z$, the composite 
\[
\Ker N^{(i)}_a \xrightarrow{\text{inclusion}} W^{(i)}_a 
\xrightarrow{Y^{(i)}_a X^{(i)}_a + k}
W^{(i)}_a \xrightarrow{\text{projection}} \Coker N^{(i)}_a
\]
is expressed as the most singular coefficient of 
$\partial_{x_i} \lambda^{(i)}_a$ and hence is invertible. 
For $a=0$, we have $W^{(i)}_0 = V^{(i)}_0/\Ker L^{(i)}_0$ and 
$Y^{(i)}_0 X^{(i)}_0 \colon W^{(i)}_0 \to W^{(i)}_0$ is the map induced from $L^{(i)}_0$.
Hence the composite
\[
\Ker N^{(i)}_0 \xrightarrow{\text{inclusion}} W^{(i)}_0 
\xrightarrow{Y^{(i)}_0 X^{(i)}_0 + k}
W^{(i)}_0 \xrightarrow{\text{projection}} \Coker N^{(i)}_0
\]
is invertible for any $k \in \Z$ if and only if 
\[
\Ker \left( L^{(i)}_0 \left( L^{(i)}_0 + k 1_{V^{(i)}_0} \right) \right) = \Ker L^{(i)}_0
\]
for any $k \in \Z$, which follows from the assumption.
Taking the direct sum, we thus see that 
the projection of $\left( Y_i X_i + k 1_{W_i} \right) |_{\Ker N_i}$
onto $\Coker N_i$ is invertible for any $k \in \Z$.

Now set $W=\bigoplus_{i=1}^m W_i$ and  
\begin{equation}\label{eq:TXY}
\begin{aligned}
T &= \bigoplus_{i=1}^m (t_i\,1_{W_i} + N_i) \in \End_\C(W), \\
X &= 
\begin{pmatrix}
X_1 & \cdots & X_m
\end{pmatrix}
\in \Hom_\C(W,V),
\\
Y &= 
\begin{pmatrix}
Y_1 \\ \vdots \\ Y_m
\end{pmatrix}
\in \Hom_\C(V,W).
\end{aligned}
\end{equation}
Let $g=(g_i) \in \wt{G}(T)$ be the element induced from $(\wh{g}_i)$, 
and write 
\[
g_i = \sum_{j \geq 0} g^{(i)}_j x_i^j, \quad  
g_i^{-1} = \sum_{j \geq 0} \bar{g}^{(i)}_j x_i^j.
\]
Define $S=\lim_{x \to \infty} A(x)$ and $(Q,P) = g \cdot (X,Y) \in \bbM(V,W)$.
Then $(V,W;\gamma):=(V,W;S,T,Q,P) \in \calH$ is stable and satisfies 
$\Phi_x(V,W;\gamma) = A$ because $\Phi_T$ is $\wt{G}(T)$-equivariant.
Furthermore, for each $i$, the blocks $Q_i, P_i$ satisfy
\[
Q_i = \sum_{j \geq 0} g^{(i)}_j X_i N_i^j, 
\quad  
P_i = \sum_{j \geq 0} N_i^j Y_i \bar{g}^{(i)}_j,
\]
and hence if we denote by $\pi_i \colon W_i \to \Coker N_i$ the projection, then
\[
\pi_i P_i Q_i |_{\Ker N_i} = 
\pi_i Y_i \bar{g}^{(i)}_0 g^{(i)}_0 X_i |_{\Ker N_i}
= \pi_i Y_i X_i |_{\Ker N_i}.
\]
Hence $\pi_i (P_i Q_i + k 1_{W_i}) |_{\Ker N_i}$ is invertible for any $k \in \Z$
and the result follows from \prpref{prp:minimal}.
\end{proof}

\begin{rmk}\label{rmk:mc}
For $\alpha \in \C$,
define a functor $\add_\alpha \colon \calS \to \calS$ by
\[
(W,B) \mapsto (W,B+y^{-1}\alpha\,dy\,1_V); \quad \varphi \mapsto \varphi.
\]
The functor 
$\mc_\alpha := \ov{\HD} \circ \add_{-\alpha} \circ \HD \colon \calS \to \calS$
introduced in \cite{DR00,Tak,Yam11}
is an additive analogue of the {\em middle convolution} 
appearing in an algorithm of Katz~\cite{Kat} and Arinkin~\cite{Ari} 
to construct all rigid meromorphic connections from the trivial rank one connection.
\end{rmk}

\section{Isomonodromic deformations}\label{sec:IMD}

Throughout this section,
we fix a nonzero finite-dimensional $\C$-vector space $V$ 
and a maximal torus $\frt \subset \gl(V)$. 
Take a basis of $V$ so that $\frt$ is identified with the standard maximal torus.

Let $\Delta$ be a contractible complex manifold (e.g.\ a polydisc).
Let $t_i \colon \Delta \to \bbP^1 \times \Delta,\ i=0,1, \dots ,m$
be holomorphic sections of the fiber bundle
$\pi \colon \bbP^1 \times \Delta \to \Delta$
such that $t_i(s) \neq t_j(s)\ (i \neq j)$ 
in each fiber $\bbP^1_s := \bbP^1 \times \{ t \}$.
In this section 
we examine the isomonodromy problem 
for families $(\nabla_s)_{s \in \Delta}$ of meromorphic connections 
on the trivial vector bundles $\calO_{\bbP^1_s} \otimes_\C V$ over $\bbP^1_s$
with poles at $t_i(s)$, $i=0,1, \dots , m$
and for the families on the Harnad dual side.

In what follows, we use the notation 
$g[A] = g A g^{-1} + dg \cdot g^{-1}$
to denote the gauge transforms.

\subsection{Isomonodromic deformations}\label{subsec:IMD}

We fix a smoothly varying standard coordinate 
$x \colon \bbP^1_s \xrightarrow{\simeq} \C \cup \{ \infty \}$
in which $t_0(s) \equiv \infty$
and re-trivialize the bundle $\bbP^1 \times \Delta$ so that $d_\Delta x =0$ for simplicity.
For $i=0,1, \dots ,m$, we put
\[
x_i \colon \bbP^1 \times \Delta \to \bbP^1; \quad (x,t) \mapsto 
\begin{cases}
1/x & (i=0), \\
x-t_i(s) & (i \neq 0),
\end{cases}
\]
which gives a coordinate on each $\bbP^1_s$ vanishing at $t_i(s)$.
For $i=0,1, \dots , m$, let $\Lambda_i$ be a smoothly varying family of irregular types 
\[
\Lambda_i(s) \in 
\frt(\calK_{t_i(s)})/\frt(\calO_{t_i(s)}) \simeq 
x_i^{-1}\frt[x_i^{-1}],
\quad s \in \Delta,
\]
such that the pole order of the difference 
of every two diagonal entries of $\Lambda_i(s)$ is constant on $\Delta$.
In particular, the reductive subgroup
\[
H_i := \{\, g \in G \mid g\Lambda_i(s)g^{-1} = \Lambda_i(s)\,\}
\]
does not depend on $s$. Let $\frh_i$ be its Lie algebra.

For $i=0,1, \dots ,m$, 
let $L_i \colon \Delta \to \frh_i$
be a holomorphic map such that 
\begin{enumerate}
\item[(E1)] for any $s \in \Delta$, $L_i(s) \in \frh_i$ is {\em non-resonant},
i.e., $\ad_{L_i(s)} \in \End \frh_i$ has no nonzero integral eigenvalues;
\item[(E2)] the $H_i$-adjoint orbit of $L_i(s)$ does not depend on $s$.
\end{enumerate}
In particular, for each $s \in \Delta$ and $i=0,1, \dots , m$, 
the connection $d_{\bbP^1} - d_{\bbP^1} \Lambda_i - L_i\,d_{\bbP^1}x_i/x_i$ 
is a normal form at $t_i(s) \in \bbP^1_s$.
We call the pair $(\bm{\Lambda},\bfL)$, where 
$\bm{\Lambda}:=(\Lambda_i)_{i=0}^m$, $\bfL := (L_i)_{i=0}^m$,  
an {\em admissible family of singularity data}.

To an admissible family of singularity data $(\bm{\Lambda},\bfL)$, 
we associate meromorphic connections
\[
\nabla^0_i = d_{\bbP^1 \times \Delta} - \calA^0_i, 
\quad
\calA^0_i := d_{\bbP^1 \times \Delta} \Lambda_i 
+ L_i \frac{d_{\bbP^1 \times \Delta}x_i}{x_i}
\]
on the trivial vector bundle $\calO_{\bbP^1 \times \Delta} \otimes_\C V$
over $\bbP^1 \times \Delta$.

\begin{dfn}\label{dfn:admissible}
The family $(\nabla_s)_{s \in \Delta}$, $\nabla_s = d_{\bbP^1_s} - A(s)$ 
of meromorphic connections on $\calO_{\bbP^1_s} \otimes_\C V$ 
is called an {\em admissible family with singularity data} 
$(\bm{\Lambda},\bfL)$ if it satisfies 
the following two conditions:
\begin{enumerate}
\item the meromorphic one-forms $A(s),\, s \in \Delta$ assemble into  
a meromorphic one-form on $\bbP^1 \times \Delta$ 
with poles on $\bigcup_{i=0}^m t_i(\Delta)$;
\item for any $i=0,1,\dots ,m$, 
there exists a family  
$\wh{g}_i \colon \Delta \to \Aut_{\C\fp{x_i}}(V \otimes_\C \C\fp{x_i})$ 
of formal power series with coefficients holomorphic on $\Delta$ such that 
for each $s \in \Delta$,
the Laurent expansion of $A(s)$ at $x_i=0$ coincides 
with the gauge transform of that of $\calA^0_i |_{\bbP^1_s}$ 
via $\wh{g}_i(s)$:
\[
A(s) = \wh{g}_i(s)[\calA^0_i |_{\bbP^1_s}]. 
\]
\end{enumerate}
\end{dfn}
It is well-known (see e.g.~\cite[Theorem~6.4]{BV}) 
that the coefficients of the power series $\wh{g}_i$ are
uniquely determined from its constant term,
whose ambiguity is 
exactly the right multiplication by map $h \colon \Delta \to H_i$ 
such that $h(s)$ commutes with $L_i(s)$ for any $s \in \Delta$.

\begin{dfn}\label{dfn:IMD}
Let $(\nabla_s)_{s \in \Delta}$ 
be an admissible family of meromorphic connections on $\calO_{\bbP^1_s} \otimes_\C V$ 
with singularity data $(\bm{\Lambda},\bfL)$.
It is said to be {\em isomonodromic} if 
there exists a {\em flat} meromorphic connection $\nabla$ 
on $\calO_{\bbP^1 \times \Delta} \otimes_\C V$ 
with poles on $\bigcup_{i=0}^m t_i(\Delta)$ 
such that $\nabla |_{\bbP^1_s} = \nabla_s$ for each $s \in \Delta$.
Such $\nabla$ is called a {\em flat extension} of $(\nabla_s)_{s \in \Delta}$.
\end{dfn}

\begin{rmk}\label{rmk:IMD}
(i) In the case where 
the most singular coefficient of each $\Lambda_i(s)$ is regular semisimple,
the above gives the isomonodromic deformations in the sense of Jimbo et al.~\cite{JMU}  

(ii) If $\Lambda_i \equiv 0$ for $i \geq 0$, $\Lambda_0$ has pole order at most $3$ 
and its most singular coefficient is constant on $\Delta$, 
then the above gives the simply-laced isomonodromy systems 
in the sense of Boalch~\cite{Boa12}.
\end{rmk}

In fact, a flat extension of an isomonodromic family 
is almost determined from the singularity data:

\begin{lmm}[cf.~{\cite[Lemma A.1]{Boa01}}]\label{lmm:flat-local}
Let $(\nabla_s)_{s \in \Delta}$ be an isomonodromic family 
of meromorphic connections on $\calO_{\bbP^1} \otimes_\C V$
with singularity data $(\bm{\Lambda},\bfL)$ 
and $\nabla = d_{\bbP^1 \times \Delta} - \calA$ a flat extension of it.
Then for each $i$ 
there exists a $\frh_i$-valued one-form $\phi_i$ on $\Delta$
satisfying the following three conditions:
\begin{enumerate}
\item $\wh{g}_i^{-1}[\calA] = \calA^0_i + \pi^* \phi_i$;
\item $d_\Delta - \phi_i$ is flat;
\item $d_\Delta L_i = [\phi_i,L_i]$.
\end{enumerate}
\end{lmm}

\begin{proof}
Fix $i$.
By the definition, the $\bbP^1$-component $A^0_i$ of $\calA^0_i$
is equal to that of $\wh{g}_i^{-1}[\calA]$. 
Put $B = \wh{g}_i^{-1}[\calA] - A^0_i$.
Then the flatness of $\nabla$ implies
\begin{equation}\label{eq:flat1}
d_{\bbP^1} B + d_\Delta A^0_i = [A^0_i, B].
\end{equation}
Take any subspace $\frh'_i \subset \frg$ complementary to $\frh_i$,
and let $B'$ be the $\frh'_i$-component of $B$.
Projecting both sides of the above equality to $\frh'_i$, we find
\[
d_{\bbP^1} B' = [A^0_i, B'].
\] 
\cite[Theorem 6.4]{BV} implies that $B'$ commutes with $\Lambda_i$; hence $B'=0$
and $B$ takes values in $\frh_i$.
\eqref{eq:flat1} reads 
\begin{equation}\label{eq:flat2}
d_{\bbP^1} B + d_\Delta A^0_i 
= [A^0_i, B] = [B,L_i] \wedge \frac{d_{\bbP^1} x_i}{x_i}.
\end{equation}
Write $B = \sum_l B_l x_i^l$, where $B_l$ are one-forms in the $\Delta$-direction.
Then the above reads
\[
d_\Delta A^0_i = [B,L_i] \wedge \frac{d_{\bbP^1} x_i}{x_i} - d_{\bbP^1} B
= \sum_l (l-\ad_{L_i})(B_l) x_i^{l-1} \wedge d_{\bbP^1}x_i.
\]
On the other hand, we have
\[
d_\Delta A^0_i = d_\Delta d_{\bbP^1}\Lambda_i 
- L_i \frac{d_\Delta x_i \wedge d_{\bbP^1}x_i}{x_i^2}
+ d_\Delta L_i \wedge \frac{d_{\bbP^1}x_i}{x_i}.
\]
Since it has no holomorphic term (as a Laurent series in $x_i$) 
and takes values in $\Ker \ad_{L_i}$,
we find $(l-\ad_{L_i})(B_l) = 0$ for $l>0$ and
\[
(l-\ad_{L_i}) \ad_{L_i} (B_l) = \ad_{L_i} (l-\ad_{L_i})(B_l) = 0
\]
for any $l$. Recall that $L_i$ is non-resonant, i.e.,  
the operator $(l-\ad_{L_i}) \in \End \frh_i$ is invertible unless $l=0$.
Hence $B_l=0$ for $l > 0$ and $\ad_{L_i} (B_l) =0$ for $l\neq 0$.
Taking the formal residue at $x_i=0$ on both sides of \eqref{eq:flat2},
we find
\[
d_\Delta L_i = [B_0,L_i].
\]
Define $\phi_i = B_0$.
Since $d_\Delta -B$ is flat, its constant term  
$d_\Delta - \phi_i$ is also flat.
To prove the rest assertion, let $\Omega^0_i$ be the $\Delta$-component of $\calA^0_i$:
\[
\Omega^0_i = d_\Delta \Lambda_i + L_i \frac{d_\Delta x_i}{x_i}.
\]
Then $\wh{g}_i^{-1}[\calA] - \calA^0_i = B - \Omega^0_i$.
We have
\begin{align*}
d_{\bbP^1} (B-\Omega^0_i) &= d_{\bbP^1} B - d_{\bbP^1} d_\Delta \Lambda_i 
+ L_i \frac{d_{\bbP^1} x_i \wedge d_\Delta x_i}{x_i^2}
\\
&= d_{\bbP^1} B + d_\Delta A^0_i - d_\Delta L_i \wedge \frac{d_{\bbP^1}x_i}{x_i}
\\
&= [B,L_i] \wedge \frac{d_{\bbP^1} x_i}{x_i} - [B_0,L_i] \wedge \frac{d_{\bbP^1}x_i}{x_i}
\\
&= [B-B_0,L_i] \wedge \frac{d_{\bbP^1} x_i}{x_i}.
\end{align*}
Since $\ad_{L_i}(B_l)=0$ for $l \neq 0$, we find $d_{\bbP^1} (B-\Omega^0_i)=0$.
Noting that $\Omega^0_i$ has no constant term in $x_i$,
we obtain $B-\Omega^0_i = \phi_i$.
\end{proof}

\begin{rmk}\label{rmk:exponent}
If the family $(\nabla_s)_{s \in \Delta}$ is isomonodromic,
then the above lemma and the contractibility of $\Delta$ show that 
for each $i$ we can find a holomorphic map $h_i \colon \Delta \to H_i$
such that $\phi_i = d_\Delta h_i \cdot h_i^{-1}$.
Then
\[
(\wh{g}_i h_i)^{-1}[\calA] = d_{\bbP^1 \times \Delta} \Lambda_i 
+ L'_i \frac{d_{\bbP^1 \times \Delta} x_i}{x_i},
\quad
L'_i = h_i^{-1} L_i h_i,
\] 
and $d_\Delta L'_i =0$, i.e., $L'_i$ is constant.
\end{rmk}

\begin{crl}\label{crl:Lax}
Let $(\nabla_s)_{s \in \Delta}$ be an isomonodromic family 
with singularity data $(\bm{\Lambda},\bfL)$ 
and $\nabla = d_{\bbP^1 \times \Delta} - \calA$ its flat extension.
Then the $\Delta$-component $\Omega$ of $\calA$ is expressed as 
\[
\Omega = \Omega' + 
\sum_{i=0}^m \left( \wh{g}_i \cdot \Omega^0_i \cdot \wh{g}_i^{-1} \right)_{i,-} 
\]
for some $\gl(V)$-valued one-form $\Omega'$ on $\Delta$,  
where $\Omega^0_i$ is the $\Delta$-component of $\calA^0_i$ 
and $( \phantom{X} )_{i,-}$  
means taking the principal part of the Laurent expansion in $x_i$.
\end{crl}

\begin{proof}
Taking the principal part of the $\Delta$-component 
on both side of the equality in 
\lmmref{lmm:flat-local}, (i), we obtain
\[
\left( \wh{g}_i^{-1} \cdot \Omega \cdot \wh{g}_i \right)_{i,-} = \Omega^0_i
\qquad (i=0,1, \dots , m).
\]
Since $\Omega$ is meromorphic, the result follows.
\end{proof}

\begin{rmk}\label{rmk:IMD2}
Let $(\nabla_s)_{s \in \Delta}$ be an isomonodromic family 
with singularity data $(\bm{\Lambda},\bfL)$ 
and $\nabla = d_{\bbP^1 \times \Delta} - \calA$ its flat extension.
According to the previous remark, 
we may assume that 
$\bfL$ is constant and $\wh{g}_i^{-1}[\calA] = \calA^0_i$
for all $i$.
Let $g \colon \Delta \to \GL(V)$ be the constant term of $\wh{g}_0$
and replace $\calA$, $\wh{g}_i$ with $g^{-1}[\calA]$, $g^{-1}\wh{g}_i$,
respectively. 
Then the relations $\wh{g}_i^{-1}[\calA] = \calA^0_i$ still hold, 
the constant term of $\wh{g}_0(s)$ is the identity, and 
$A(s) = \calA|_{\bbP^1_s}$ has the same most singular term at $x=\infty$
as $d_{\bbP^1}\Lambda_0(s)$.
In this situation one can modify the above proof of \crlref{crl:Lax} 
to obtain 
\[
\Omega = \left( \wh{g}_0 \cdot \Omega^0_0 \cdot \wh{g}_0^{-1} \right)_{0,\leq 0} + 
\sum_{i=1}^m \left( \wh{g}_i \cdot \Omega^0_i \cdot \wh{g}_i^{-1} \right)_{i,-},
\]
where $( \phantom{X} )_{i,\leq 0}$  
means taking the non-positive degree part of the Laurent expansion in $x_i$.
Note that it depends only on the family $(d_{\bbP^1_s}-A(s))_{s \in \Delta}$.
Indeed, $\wh{g}_0(s)$ is uniquely determined from $A(s)$ and 
the ambiguity of the choice of $\wh{g}_i$ for each $i \geq 1$
is only the right multiplication by map $h_i \colon \Delta \to H_i$ commuting with $L_i$,
while $h_i \Omega^0_i h_i^{-1} = \Omega^0_i$ for such $h_i$.
Hence the flatness condition for $\calA$ gives a system of non-linear differential equations 
for $A(s)$, the ``isomonodromy equation''.
\end{rmk}

\subsection{Main theorem}\label{subsec:main}

In what follows we fix 
an admissible family of singularity data $(\bm{\Lambda},\bfL)$.
As in \secref{subsec:normal}, 
we label the nonzero diagonal entries of $\Lambda_i$
as $\lambda^{(i)}_1, \lambda^{(i)}_2, \dots , \lambda^{(i)}_{d_i}$
and set $\lambda^{(i)}_0 \equiv 0$ for convenience.
For each $i$, we then have a decomposition 
$V = \bigoplus_{a=0}^{d_i} V^{(i)}_a$, where
$\Lambda_i |_{V^{(i)}_a} = \lambda^{(i)}_a\,1_{V^{(i)}_a}$. 
(If any diagonal entry of $\Lambda_i$ is nonzero then $V^{(i)}_0 = \{ 0 \}$.)
We assume the following three conditions:
\begin{enumerate}
\item[(A1)] $\Lambda_0=0$;
\item[(A2)] the pole order of $\lambda^{(i)}_a - \lambda^{(i)}_b$ 
is constant on $\Delta$ for each $i,a,b$;
\item[(A3)] each $L_i$ is constant on $\Delta$.
\end{enumerate}
The second condition is non-trivial unless 
$\Lambda_i$ has zero in its diagonal entries,
and is assumed so that the rank of the Fourier-Laplace transform is constant on $\Delta$.
The third condition is not essential (see \rmkref{rmk:exponent})
but we assume it to simplify the arguments. 

According to the decomposition $H_i = \prod_{a=0}^{d_i} \GL(V^{(i)}_a)$,
we express 
\[
L_i = \bigoplus_{a=0}^{d_i} L^{(i)}_a,
\quad L^{(i)}_a \in \gl(V^{(i)}_a).
\]
For each $i,a$, we put 
\[
A^0_{i,a}=d_{\bbP^1} \lambda^{(i)}_a\,1_{V^{(i)}_a} + L^{(i)}_a\,d_{\bbP^1} x_i /x_i
\]
and let
\[
(V^{(i)}_a, W^{(i)}_a; 0, N^{(i)}_a, X^{(i)}_a, Y^{(i)}_a) 
= \kappa_{x_i}(V^{(i)}_a, A^0_{i,a}) \in \calH.
\]
We have already calculated $W^{(i)}_a$; see \secref{subsec:normal}.
By assumption (A3), each $W^{(i)}_a$ does not depend on $s \in \Delta$. 
We set $W = \bigoplus_{i=1}^m \bigoplus_{a=0}^{d_i} W^{(i)}_a$ 
and state the main result of this paper.

\begin{thm}\label{thm:main}
Let $(\nabla_s)_{s \in \Delta}$, $\nabla_s = d_{\bbP^1} -A(s)$ be an admissible family 
of meromorphic connections on $\calO_{\bbP^1} \otimes_\C V$ with 
singularity data $(\bm{\Lambda},\bfL)$ satisfying assumptions~{\rm (A1--3)}
and $W$ the vector space defined above.
Then there exists a family 
$(\nabla_s^\vee)_{s \in \Delta}$, $\nabla_s^\vee = d_{\bbP^1} - B(s)$
of meromorphic connections on $\calO_{\bbP^1} \otimes_\C W$ satisfying 
the following three conditions:
\begin{enumerate}
\item $(W,B(s)) \simeq \HD(V,A(s))$ for any $s \in \Delta$ 
and $B(s)$ depends smoothly on $s \in \Delta$;
\item if $(\nabla_s)_{s \in \Delta}$ is isomonodromic, 
then there exist $\gl(W)$-valued one-forms 
$\Theta, \Xi$ on $\Delta$ such that the meromorphic connection
\[
\nabla^\vee = d_{\bbP^1 \times \Delta} - B - (\Theta y + \Xi)
\]
is flat, where $B$ is the meromorphic one-form on $\bbP^1 \times \Delta$ 
induced from $B(s)$, $s \in \Delta$ (so it has no $\Delta$-component
and $B |_{\bbP^1_s} = B(s)$ for all $s \in \Delta$);
\item conversely, if there exists a flat meromorphic connection
$\nabla^\vee$ of the above form, 
and furthermore if $(V,A(s)) \in \calS$ is irreducible for any $s \in \Delta$
or $L_0$ is invertible, then $(\nabla_s)_{s \in \Delta}$ is isomonodromic.
\end{enumerate}
\end{thm}

The rest of this section is devoted to prove the above theorem.

\subsection{Construction of the dual family}\label{subsec:dual}

First, we associate 
to an admissible family $(\nabla_s)_{s \in \Delta}$ of meromorphic connections 
the dual family $(\nabla^\vee_s)_{s \in \Delta}$ 
satisfying condition (i) in \thmref{thm:main}.

For each $i=1,2, \dots , m$, set $W_i = \bigoplus_{a=0}^{d_i} W^{(i)}_a$ 
and  
\[
N_i=\bigoplus_{a=0}^{d_i} N^{(i)}_a, \quad 
X_i = \bigoplus_{a=0}^{d_i} X^{(i)}_a, \quad 
Y_i = \bigoplus_{a=0}^{d_i} Y^{(i)}_a.
\]
Then 
\[
X_i (x_i\,1_{W_i} - N_i)^{-1} Y_i\, d_{\bbP^1}x_i 
= A^0_i = d_{\bbP^1} \Lambda_i + L_i\,d_{\bbP^1} x_i /x_i,
\]
and $(V,W_i;0,N_i,X_i,Y_i) \in \calH$ is stable at each $s \in \Delta$
(because it is isomorphic to $\kappa_{x_i}(V,A^0_i)$).

Let $(\nabla_s)_{s \in \Delta}$, 
$\nabla_s = d_{\bbP^1_s} - A(s)$ be an admissible family 
of meromorphic connections with singularity data $(\bm{\Lambda},\bfL)$
and $A$ the induced meromorphic one-form on $\bbP^1 \times \Delta$. 
Let $g=(g_i) \in \wt{G}(T)$ be the element 
induced from $(\wh{g}_i)$. 
Define $T, X, Y$ as in \eqref{eq:TXY} and set 
\begin{equation}\label{eq:QP}
(Q,P) = g \cdot (X,Y) \in \bbM(V,W).
\end{equation}
(Note that $T, Q, P$ depend on $s \in \Delta$.)
Then 
\[
Q (x 1_W - T)^{-1} P\, d_{\bbP^1} x = A,
\]
and $(V,W;0,T,Q,P) \in \calH$ is stable at each $s \in \Delta$.
Define
\[
\nabla^\vee_s = d_{\bbP^1} - B(s), \quad 
B(s) = - \left( T(s) + \frac{P(s)Q(s)}{y} \right) d_{\bbP^1}y \quad (s \in \Delta).
\]
Then $(W,B(s)) \simeq \HD(V,A(s))$ for all $s \in \Delta$.
We show that this family $(\nabla^\vee_s)_{s \in \Delta}$ 
satisfies conditions (ii), (iii) 
in \thmref{thm:main}.

\begin{rmk}\label{rmk:Wood}
If the most singular coefficient of each $\Lambda_i$ is invertible,
then our dual family $(\nabla_s^\vee)_{s \in \Delta}$ is 
isomorphic to Woodhouse's~\cite{Woo07}.
\end{rmk}

\subsection{Construction of the one-form $\Theta$}\label{subsec:Theta}

We construct a $\gl(W)$-valued one-form $\Theta$ on $\Delta$ appearing in 
\thmref{thm:main}, (ii);
in fact, it depends only on the singularity data $(\bm{\Lambda},\bfL)$.

\begin{lmm}\label{lmm:Theta}
For each $i$, there exists a unique 
$\Ker \ad_{N_i}$-valued one-form $\Theta_i$ on $\Delta$ such that
\[
\Omega^0_i = X_i (x_i\,1_{W_i} - N_i)^{-1} \Theta_i Y_i, \quad \Theta_i \wedge \Theta_i =0.
\]
\end{lmm}

\begin{proof}
For $i=1,2, \dots , m$ and $a=0,1, \dots ,d_i$, let 
$k_{i,a}$ be the pole order of $\partial_{x_i} \lambda^{(i)}_a$ 
and set $\calR_{i,a} = \C[x_i]/(x_i^{k_{i,a}})$.
We have to find for each $i$ 
a one-form $\Theta_i = \bigoplus_a \Theta^{(i)}_a$ with $\Theta^{(i)}_a$ 
taking values in $\End_{\calR_{i,a}}(W^{(i)}_a)$ such that
\begin{align*}
X^{(i)}_0 \cdot \Theta^{(i)}_0 \cdot Y^{(i)}_0 &= L^{(i)}_0 \frac{d_\Delta x_i}{x_i}, 
\\ 
\wt{X}^{(i)}_a \cdot \Theta^{(i)}_a \cdot \wt{Y}^{(i)}_a &= 
x_i^{k_{i,a}} d_\Delta \lambda^{(i)}_a \,1_{W^{(i)}_a} 
+ x_i^{k_{i,a}-1} L^{(i)}_a\,d_\Delta x_i \otimes 1_{\calR_{i,a}}
\quad (a \neq 0).
\end{align*}
The first equality is satisfied by $\Theta^{(i)}_0 := d_\Delta x_i\,1_{W^{(i)}_0}$.
Since $\wt{X}^{(i)}_a$ is invertible and $\wt{Y}^{(i)}_a$ is identity for $a \neq 0$,
the second equality is satisfied by 
\begin{equation}\label{eq:Theta}
\Theta^{(i)}_a := 
(\wt{X}^{(i)}_a)^{-1} \left( x_i^{k_{i,a}} d_\Delta \lambda^{(i)}_a \,1_{W^{(i)}_a} 
+ x_i^{k_{i,a}-1} L^{(i)}_a\,d_\Delta x_i \otimes 1_{\calR_{i,a}} \right).
\end{equation}
Note that 
$\wt{X}^{(i)}_a$ lies, and $\wt{X}^{(i)}_a \Theta^{(i)}_a \wt{Y}^{(i)}_a$ takes values, 
in $\calR_{i,a} \cdot 1_{W^{(i)}_a} + L^{(i)}_a \otimes (\calR_{i,a} \cdot 1_{\calR_{i,a}})$,
whose elements commute with one another.
Hence $\Theta_i \wedge \Theta_i =0$.
The uniqueness follows from \lmmref{lmm:stability}.
\end{proof}

\begin{lmm}\label{lmm:AHHP}
If $(\nabla_s)_{s \in \Delta}$ is isomonodromic 
with flat extension 
$\nabla = d_{\bbP^1 \times \Delta} - \calA$,
then $\Theta = \bigoplus_i \Theta_i$ satisfies
$Q (x 1_W - T)^{-1} \Theta P = \Omega - \Omega_\infty$,
where $\Omega$ is the $\Delta$-component of $\calA$ 
and $\Omega_\infty := \Omega |_{z=\infty}$.
\end{lmm}

\begin{proof}
Immediately follows from \lmmref{lmm:Theta} and \crlref{crl:Lax}.
\end{proof}

Hence the flat extension $\nabla = d_{\bbP^1 \times \Delta} - \calA$
is described as
\[
\calA = \Omega_\infty + Q(x 1_W -T)^{-1}(d_{\bbP^1}x + \Theta)P,
\]
which we call the {\em extended AHHP representation}.

\begin{rmk}\label{rmk:normalize}
The flatness condition of the above $\nabla$ 
implies that $d_\Delta - \Omega_\infty$ is flat;
hence there is a holomorphic map $g \colon \Delta \to \GL(V)$
such that 
\[
g[\calA] = g Q(x 1_W -T)^{-1}(d_{\bbP^1}x + \Theta)P g^{-1}.
\]
In other words, we can normalize the isomonodromic family with flat extension 
so that $\Omega_\infty=0$.
\end{rmk}

\begin{example}\label{ex:Schlesinger1}
Suppose $\bm{\Lambda}=0$.
As shown in the proof of \lmmref{lmm:Theta}, 
the one-form $\Theta$ is then given by 
\[
\Theta_i = d_\Delta x_i\,1_{W_i} = - d_\Delta t_i\,1_{W_i}
\quad (i=1,2, \dots , m),
\]
i.e., $\Theta = - d_\Delta T$.
Hence if $(\nabla_s)$ is isomonodromic with flat extension 
$\nabla = d_{\bbP^1 \times \Delta} - \calA$,
the AHHP representation is simply expressed as
\[
\calA = \Omega_\infty + Q(x 1_W -T)^{-1}(d_{\bbP^1}x - d_\Delta T)P 
= \Omega_\infty + Q\,d_{\bbP^1 \times \Delta}\log(x 1_W -T)P.
\]
\end{example}

\subsection{Existence of a one-form $\Xi$}\label{subsec:Xi}

Next we find a one-form $\Xi$ appearing in \thmref{thm:main}, (ii). 
We start with the following elementary lemma:

\begin{lmm}\label{lmm:adjoint}
Let $V$ be a finite-dimensional $\C$-vector space and $l \in \Z_{>0}$.
Put $W = V \otimes_\C \left( \C[z]/(z^l) \right)$ 
and let $N \in \End_\C(W)$ be the multiplication by $\zeta$.
Then $X \in \End_\C(W)$ is contained in $\range \ad_N$
if and only if 
\[
\sum_{j=1}^l N^{l-j} X N^{j-1} =0.
\] 
\end{lmm}

\begin{proof}
Define a linear map $\varphi \colon \End_\C(W) \to \End_\C(W)$ by
\[
\varphi(X) = \sum_{j=1}^l N^{l-j} X N^{j-1}.
\]
We first show $\range \ad_N \subset \Ker \varphi$.
For $X \in \End_\C(W)$, 
\begin{align*}
\sum_{j=1}^l N^{l-j} [N,X] N^{j-1} 
&= \sum_{j=1}^l N^{l-j+1} X N^{j-1}
- \sum_{j=1}^l N^{l-j} X N^j 
\\
&= N^l X - X N^l =0.
\end{align*}
Next we show $\rank \varphi = l\dim V = \dim \Ker \ad_N$.
According to the decomposition
\[
\End_\C(W) = \bigoplus_{i,j =0}^{l-1} \Hom_\C(V \otimes \C z^j, V \otimes \C z^i),
\]
we write each $X \in \End_\C(W)$ as 
\[
X=(X_{ij}), \quad
X_{ij} \in \Hom_\C(V \otimes \C z^j, V \otimes \C z^i) \simeq \End_\C(V).
\]
Then a direct calculation shows
\[
\varphi(X)_{ij} = \sum_{a=l-i}^{l-j} X_{a+i-l,a+j-1} 
= \sum_{a=0}^{i-j} X_{a,a+j-i+l-1}.
\]
Hence
\begin{equation*}
\range \varphi = 
\rset{Y=(Y_{ij}) \in \End_\C(W)}%
{
\begin{aligned}
Y_{ij}&=0 & &(i<j), \\
Y_{ij}&=Y_{ab}& &(i-j=a-b)
\end{aligned}
}.
\end{equation*}
This implies $\rank \varphi = l \dim V$.
\end{proof}

\begin{lmm}\label{lmm:Xi}
Assume that $(\nabla_s)_{s \in \Delta}$ is isomonodromic 
with flat extension $\nabla = d_{\bbP^1 \times \Delta} - \calA$ and 
let $\calA = \Omega_\infty + Q(x 1_W -T)^{-1}(d_{\bbP^1}x + \Theta)P$
be the extended AHHP representation.
Then there exists a unique $\gl(W)$-valued one-form $\Xi$ on $\Delta$ 
such that $[\Xi,T] = d_\Delta T + \Theta + [PQ, \Theta]$ and that $Q, P$ satisfies 
the differential equations
\begin{equation}\label{eq:diff-QP}
d_\Delta Q = \Omega_\infty Q - Q \Xi, \quad
d_\Delta P = -P \Omega_\infty + \Xi P.
\end{equation}
Furthermore, it satisfies
\[
d_\Delta \Theta - [\Theta , \Xi] =0, 
\quad 
d_\Delta \Xi - \Xi \wedge \Xi =0.
\]
\end{lmm}

\begin{proof}
We first show that $d_\Delta T + \Theta + [PQ, \Theta]$ takes values in $\range \ad_T$.
As $\ad_T$ preserves each $\Hom_\C(W_j,W_i)$ and is invertible on it 
if $i \neq j$, it is sufficient to show that 
$d_\Delta t_i\,1_{W_i} + \Theta_i + [P_i Q_i, \Theta_i]$ takes values in 
$\range \ad_{N_i}$ for each $i$.
Furthermore, under the notation used in \eqref{eq:action-Q} and \eqref{eq:action-P} 
we have 
\begin{align*}
P_i Q_i &= \sum_{j,l \geq 0} N_i^j Y_i\, \bar{g}^{(i)}_j\, g^{(i)}_l X_i N_i^l \\
&\equiv \sum_{j,l \geq 0} Y_i\, \bar{g}^{(i)}_j\, g^{(i)}_l X_i N_i^{j+l} \equiv Y_i X_i
\quad (\text{mod}\ \range \ad_{N_i}).
\end{align*}
Therefore we may replace the term $[P_i Q_i, \Theta_i]$ with $[Y_i X_i, \Theta_i]$
(note that $\Theta_i$ commutes with $N_i$).
Recall that $X_i, Y_i, \Theta_i$ respect the decompositions 
$V = \bigoplus_a V^{(i)}_a$, $W_i = \bigoplus_a W^{(i)}_a$
and the components $X^{(i)}_a, Y^{(i)}_a, \Theta^{(i)}_a$ are 
explicitly given in the proof of \lmmref{lmm:Theta}.
We have $\Theta^{(i)}_0 = d_\Delta x_i\,1_{W^{(i)}_0} = -d_\Delta t_i\,1_{W^{(i)}_0}$
and thus
\[
\left( d_\Delta t_i\,1_{W_i} + \Theta_i + [Y_i X_i, \Theta_i] \right)|_{W^{(i)}_0} 
= [Y^{(i)}_0 X^{(i)}_0, -d_\Delta t_i\,1_{W^{(i)}_0}] =0.
\]
For $a \neq 0$, from the definition \eqref{eq:Theta} of $\Theta^{(i)}_a$ 
we see that $\Theta^{(i)}_a + d_\Delta t_i\, 1_{W^{(i)}_a}$ takes values in 
$x_i \End_{\calR_{i,a}}(W^{(i)}_a)$.
The obvious identity $[x_i \partial_{x_i}, x_i] = x_i$ 
in $\End_\C(\calR_{i,a})$ shows $x_i \End_{\calR_{i,a}}(W^{(i)}_a) \subset \range \ad_{N^{(i)}_a}$.
Furthermore, we find
\begin{align*}
\sum_{j=1}^{k_{i,a}} (N^{(i)}_a)^{k_{i,a}-j} 
&[Y^{(i)}_a X^{(i)}_a, \Theta^{(i)}_a] (N^{(i)}_a)^{j-1} \\
&= \left[ \sum_{j=1}^{k_{i,a}} (N^{(i)}_a)^{k_{i,a}-j} 
Y^{(i)}_a X^{(i)}_a (N^{(i)}_a)^{j-1}, \Theta^{(i)}_a \right] \\
&= [\wt{Y}^{(i)}_a \wt{X}^{(i)}_a, \Theta^{(i)}_a] =0,
\end{align*}
which together with \lmmref{lmm:adjoint} 
implies $[Y^{(i)}_a X^{(i)}_a, \Theta^{(i)}_a]$ takes values in 
$\range \ad_{N^{(i)}_a}$.

Thus we can take a $\gl(W)$-valued one-form $\Xi$ on $\Delta$ 
such that 
\[
d_\Delta T + \Theta + [PQ, \Theta] = [\Xi,T] = [x 1_W -T, \Xi].
\]
We substitute it into the following formula
\begin{equation}\label{eq:flat1-1}
\begin{aligned}
d_\Delta A &+ d_{\bbP^1} \Omega - [A, \Omega] \\
&=(d_\Delta Q - \Omega_\infty Q) \wedge (x 1_W -T)^{-1}P\,d_{\bbP^1}x \\
&\quad -Q(x 1_W -T)^{-1} d_{\bbP^1}x \wedge (d_\Delta P + P \Omega_\infty) \\
&\quad + Q(x 1_W -T)^{-1}(d_\Delta T + \Theta + [PQ, \Theta])(x 1_W -T)^{-1}P \wedge d_{\bbP^1}x,
\end{aligned}
\end{equation}
which is verified by substituting $A=Q(x 1_W -T)^{-1}P\, d_{\bbP^1}x$ 
and $\Omega = \Omega_\infty + Q(x 1_W -T)^{-1}\Theta P$. 
By the flatness condition, we then obtain
\begin{equation}\label{eq:flat1-1'}
\begin{aligned}
0 
&=(d_\Delta Q - \Omega_\infty Q + Q \Xi) \wedge (x 1_W -T)^{-1}P\,d_{\bbP^1}x \\
&\quad -Q(x 1_W -T)^{-1} d_{\bbP^1}x \wedge (d_\Delta P + P \Omega_\infty -\Xi P),
\end{aligned}
\end{equation}
which together with \lmmref{lmm:stability} implies that 
there exists a unique $\Ker \ad_T$-valued one-form $\Xi'$ on $\Delta$ 
such that
\[
d_\Delta Q - \Omega_\infty Q + Q \Xi = Q \Xi', \quad 
d_\Delta P + P \Omega_\infty - \Xi P = -\Xi' P.
\]
We may now replace $\Xi$ with $\Xi - \Xi'$
so that it satisfies all the desired conditions.
The uniqueness of $\Xi$ follows from \lmmref{lmm:stability}.

The flatness condition also implies $d_\Delta \Omega - \Omega \wedge \Omega =0$.
The restriction of it to $z=\infty$ yields 
$d_\Delta \Omega_\infty - \Omega_\infty \wedge \Omega_\infty =0$.
Furthermore, we have
\begin{equation}\label{eq:omega}
\begin{aligned}
d_\Delta \Omega - \Omega \wedge \Omega &=
d_\Delta \Omega_\infty - \Omega_\infty \wedge \Omega_\infty 
+ Q(x 1_W -T)^{-1} d_\Delta \Theta \, P \\
&\quad + (d_\Delta Q - \Omega_\infty Q) \wedge (x 1_W -T)^{-1} \Theta P \\
&\quad - Q(x 1_W -T)^{-1}\Theta \wedge (d_\Delta P + P\Omega_\infty ) \\
&\quad + Q(x 1_W -T)^{-1} (d_\Delta T - \Theta PQ) \wedge \Theta  
(x 1_W -T)^{-1} P.
\end{aligned}
\end{equation}
Substituting the differential equations on $\Omega_\infty, Q, P$ obtained so far into the above, 
we obtain
\begin{equation}
\begin{aligned}
d_\Delta \Omega - \Omega \wedge \Omega
&= Q(x 1_W -T)^{-1} d_\Delta \Theta \, P \\
&\quad - Q \Xi \wedge (x 1_W -T)^{-1} \Theta P 
- Q(x 1_W -T)^{-1}\Theta \wedge \Xi P \\
&\quad + Q(x 1_W -T)^{-1} (d_\Delta T - \Theta PQ) \wedge \Theta  
(x 1_W -T)^{-1} P,
\end{aligned}
\end{equation}
and further substituting the equality
\[
d_\Delta T - \Theta PQ = [x 1_W -T, \Xi] - \Theta - PQ \Theta,
\]
we obtain
\begin{equation}\label{eq:omega2}
\begin{aligned}
d_\Delta \Omega - \Omega \wedge \Omega 
&= Q(x 1_W -T)^{-1}(d_\Delta \Theta -\Theta \wedge \Xi - \Xi \wedge \Theta ) P \\
&\quad - Q(x 1_W -T)^{-1} (\Theta + PQ \Theta ) \wedge \Theta (x 1_W -T)^{-1} P.
\end{aligned}
\end{equation}
The second term on the right hand side is zero because $\Theta \wedge \Theta =0$.
Furthermore, 
\begin{align*}
[d_\Delta \Theta -[\Theta, \Xi],T] &=
d_\Delta [\Theta ,T] + [\Theta ,d_\Delta T] - [[\Theta , \Xi],T] \\
&= -[\Theta ,[ \Xi,T]] - [\Xi ,[\Theta ,T]] \\
&= -[\Theta , d_\Delta T + \Theta + [PQ ,\Theta ]] \\
&= -[\Theta , [PQ ,\Theta ]] = \frac12 [[\Theta ,\Theta],PQ] =0. 
\end{align*}
Therefore \lmmref{lmm:stability} shows $d_\Delta \Theta -[\Theta , \Xi] =0$.

We finally show $d_\Delta \Xi - \Xi \wedge \Xi =0$. 
Taking the exterior derivative in the $\Delta$-direction
of $d_\Delta Q = \Omega_\infty Q - Q \Xi$,
we find
\begin{align*}
0 = d_\Delta^2 Q 
&= d_\Delta \Omega_\infty  \cdot Q - \Omega_\infty \wedge d_\Delta Q 
- d_\Delta Q \wedge \Xi - Q d_\Delta \Xi
\\
&= d_\Delta \Omega_\infty  \cdot Q - \Omega_\infty \wedge (\Omega_\infty Q - Q \Xi)
- (\Omega_\infty Q - Q \Xi) \wedge \Xi - Q d_\Delta \Xi
\\
&= (d_\Delta \Omega_\infty - \Omega_\infty \wedge \Omega_\infty )Q 
-Q (d_\Delta \Xi - \Xi \wedge \Xi ) 
\\
&= -Q (d_\Delta \Xi - \Xi \wedge \Xi ).
\end{align*}
On the other hand, we have
\begin{equation}
\label{eq:PQ}
\begin{aligned}
d_\Delta (PQ)  &= d_\Delta P \cdot Q + P \cdot d_\Delta Q \\
&= (-P \Omega_\infty + \Xi P) Q + P( \Omega_\infty Q - Q \Xi ) 
= [\Xi,PQ],
\end{aligned}
\end{equation}
and thus
\begin{align*}
[d_\Delta \Xi, T] 
&= d_\Delta [\Xi ,T] + [\Xi ,d_\Delta T] \\
&= d_\Delta (\Theta + [PQ, \Theta])+ [\Xi ,d_\Delta T]  \\
&= [\Theta , \Xi] + d_\Delta [PQ, \Theta]
+ [\Xi ,d_\Delta T] \\
&= [\Theta , \Xi] + [[\Xi ,PQ], \Theta ] + [PQ, [\Theta , \Xi ]] 
+ [\Xi ,d_\Delta T] \\
&= [\Theta + [PQ, \Theta ] + d_\Delta T, \Xi ] \\
&= [[\Xi , T], \Xi ] = [\Xi \wedge \Xi , T].
\end{align*}
Hence $d_\Delta \Xi - \Xi \wedge \Xi$ commutes with $T$ and satisfies
\[
Q(x 1_W -T)^{-1}(d_\Delta \Xi - \Xi \wedge \Xi)P=0.
\]
\lmmref{lmm:stability} shows 
$d_\Delta \Xi - \Xi \wedge \Xi =0$.
\end{proof}

\begin{example}\label{ex:Schlesinger2}
In the situation of Example~\ref{ex:Schlesinger1},
the condition
\[
[\Xi,T] = d_\Delta T + \Theta + [PQ, \Theta] = -[PQ, d_\Delta T]
\]
determines the $\Hom_\C(W_j,W_i)$-block $\Xi_{ij}$ of $\Xi$ for each distinct $i,j$: 
\[
\Xi_{ij} =  -P_i Q_j\, d_\Delta \log (t_i-t_j).
\]
We show that the block diagonal part of $\Xi$ can be eliminated 
by the $G_T$-action on $\bbM^{st}$.
For a holomorphic map 
$f \colon \Delta \to G_T = \prod_{i=1}^m \GL(W_i)$,
we have
\begin{align*}
d_\Delta (Q f^{-1}) &= d_\Delta Q \cdot f^{-1} - Q f^{-1} d_\Delta f \cdot f^{-1} \\
&= (\Omega_\infty Q - Q \Xi) f^{-1} - Q f^{-1} d_\Delta f \cdot f^{-1} \\
&= \Omega_\infty (Q f^{-1}) - (Q f^{-1}) f[\Xi],
\end{align*}
and similarly 
\[
d_\Delta (f P) = - (f P) \Omega_\infty + f[\Xi](f P).
\]
Thanks to the flatness condition $d_\Delta \Xi - \Xi \wedge \Xi =0$,
we can take $f$ so that $f[\Xi]$ is block off-diagonal.
\end{example}

\subsection{Proof of the main theorem}\label{subsec:proof}

Now we prove \thmref{thm:main}.

\begin{proof}[Proof of \thmref{thm:main}]
We first show that the dual family $(\nabla^\vee_s)$ defined in 
\secref{subsec:dual} satisfies condition (ii). 
Assume that $(\nabla_s)_{s \in \Delta}$ is isomonodromic with flat extension 
$\nabla = d_{\bbP^1 \times \Delta} - \calA$.
Let $\calA = \Omega_\infty + Q(x 1_W - T)^{-1}(d_{\bbP^1}x + \Theta)P$
be the extended AHHP representation and 
$\Xi$ as in \lmmref{lmm:Xi}.
We then show that
\[
\nabla^\vee=d_{\bbP^1 \times \Delta} - B - \Omega^\vee, \quad 
\Omega^\vee := \Theta y + \Xi
\]
is flat. A direct calculation shows
\begin{equation}\label{eq:flat2-1}
\begin{aligned}
d_\Delta B + d_{\bbP^1} \Omega^\vee - [B, \Omega^\vee] 
&= [\Theta ,T] \wedge y d_{\bbP^1}y \\
&\quad - (d_\Delta T + \Theta + [PQ,\Theta] -[\Xi,T]) \wedge d_{\bbP^1}y \\
&\quad - (d_\Delta (PQ) -[\Xi,PQ]) \wedge y^{-1}d_{\bbP^1}y.
\end{aligned}
\end{equation}
Since $\Theta_i$ commutes with $N_i$, their direct sum 
$\Theta = \bigoplus_i \Theta_i$ commutes with $T$.
Also, \lmmref{lmm:Xi} and equality~\eqref{eq:PQ} imply 
that the second and third terms in the right hand side 
is zero.
Hence $d_\Delta B + d_{\bbP^1} \Omega^\vee - [B, \Omega^\vee] =0$.
We also have
\begin{equation}\label{eq:flat2-2}
d_\Delta \Omega^\vee - \Omega^\vee \wedge \Omega^\vee 
= - \Theta \wedge \Theta \, y^2 + (d_\Delta \Theta - [\Theta , \Xi]) y 
+(d_\Delta \Xi - \Xi \wedge \Xi).
\end{equation}
Lemmas~\ref{lmm:Theta} and \ref{lmm:Xi} imply that the above is zero.
Hence $\nabla^\vee$ is flat.

Next we show that $(\nabla^\vee_s)$ satisfies condition (iii).
Assume that the object $(V, A(s)) \in \calS$ 
is irreducible for any $s \in \Delta$, or $L_0$ is invertible.
Assume further that there exist $\gl(W)$-valued one-forms $\Theta, \Xi$ on $\Delta$
such that the meromorphic connection
\[
\nabla^\vee =d_{\bbP^1 \times \Delta} - B -(\Theta y + \Xi)
\]
is flat.
Then we show that $(\nabla_s)_{s \in \Delta}$ is isomonodromic.
Equalities~\eqref{eq:flat2-1} and \eqref{eq:flat2-2} imply 
$d_\Delta (PQ) -[\Xi,PQ] =0$ and 
\begin{equation}\label{eq:five}
\begin{aligned}
&[T,\Theta]=0,\quad \Theta \wedge \Theta =0,\quad d_\Delta \Theta = [\Theta,\Xi], \\
&[T,\Xi]+d_\Delta T+\Theta+[PQ,\Theta]=0,\quad d_\Delta \Xi = \Xi \wedge \Xi.
\end{aligned}
\end{equation}
We rewrite the third equality as
\[
(d_\Delta P - \Xi P)Q = -P(d_\Delta Q + Q\Xi ).
\]
By the first assumption, $P$ is injective and $Q$ is surjective 
(note that $-QP = \res_\infty A$ is contained in the adjoint orbit of $L_0$).
Hence there exist $\gl(V)$-valued one-forms $\Omega_\infty, \Omega'_\infty$ such that
$d_\Delta P - \Xi P = P \Omega'_\infty$, $d_\Delta Q + Q\Xi = \Omega_\infty Q$.
Substituting them into the above equality, we obtain
\[
P (\Omega_\infty + \Omega'_\infty )Q =0,
\]
which implies $\Omega'_\infty =-\Omega_\infty$.
Now we define a meromorphic connection $\nabla = d_{\bbP^1 \times \Delta} -\calA$ on 
$\calO_{\bbP^1 \times \Delta} \otimes_\C V$ by
\[
\calA = A + \Omega, \quad \Omega = \Omega_\infty + Q(x 1_W - T)^{-1} \Theta P, 
\]
and we show that it is flat.
First, the substitution of the equality 
$d_\Delta T + \Theta + [PQ,\Theta] = [\Xi,T] = [x 1_W - T, \Xi]$ 
into \eqref{eq:flat1-1} yields
\begin{equation}
\begin{aligned}
d_\Delta A &+ d_{\bbP^1} \Omega - [A, \Omega] \\
&=(d_\Delta Q - \Omega_\infty Q + Q \Xi) \wedge (x 1_W -T)^{-1}P\,d_{\bbP^1}x \\
&\quad -Q(x 1_W -T)^{-1} d_{\bbP^1}x \wedge (d_\Delta P + P \Omega_\infty -P \Xi) =0.
\end{aligned}
\end{equation}
Next, taking the exterior derivative (in the $\Delta$-direction)
of the equality $d_\Delta Q = \Omega_\infty Q - Q \Xi$, 
we find
\begin{align*}
0 = d_\Delta^2 Q 
&= d_\Delta \Omega_\infty  \cdot Q - \Omega_\infty \wedge d_\Delta Q 
- d_\Delta Q \wedge \Xi - Q d_\Delta \Xi
\\
&= d_\Delta \Omega_\infty  \cdot Q - \Omega_\infty \wedge (\Omega_\infty Q - Q \Xi)
- (\Omega_\infty Q - Q \Xi) \wedge \Xi - Q d_\Delta \Xi
\\
&= (d_\Delta \Omega_\infty - \Omega_\infty \wedge \Omega_\infty )Q 
-Q (d_\Delta \Xi - \Xi \wedge \Xi ) 
\\
&= (d_\Delta \Omega_\infty - \Omega_\infty \wedge \Omega_\infty )Q.
\end{align*}
Since $Q$ is surjective, we obtain $d_\Delta \Omega_\infty - \Omega_\infty \wedge \Omega_\infty =0$.
Hence equality \eqref{eq:omega} reduces to equality \eqref{eq:omega2},
which together with $\Theta \wedge \Theta =0$, $d_\Delta \Theta - [\Theta,\Xi]=0$ 
imply that $\nabla$ is flat.
\end{proof}

\subsection{Application to the additive middle convolution}

\thmref{thm:main} has the following corollary on 
the additive middle convolution (see \rmkref{rmk:mc}),
which generalizes the result of Haraoka-Filipuk~\cite{HF}:

\begin{crl}\label{crl:mc}
Let $(\nabla_s)_{s \in \Delta}$, $\nabla_s = d_{\bbP^1} -A(s)$ 
be an isomonodromic family of meromorphic connections on 
$\calO_{\bbP^1} \otimes_\C V$ with singularity data $(\bm{\Lambda},\bfL)$ 
satisfying assumptions~{\rm (A1--3)}.
Assume that $L_0$ is invertible.
%
Then there exists a pair $(\calV^\alpha,\nabla^\alpha)$ of a holomorphic vector bundle 
and a flat meromorphic connection on $\bbP^1 \times \Delta$ 
with poles on $\bigcup_{i=0}^m t_i(\Delta)$ 
such that $(\calV^\alpha,\nabla^\alpha) |_{\bbP^1_s}$ is isomorphic to 
the connection given by $\mc_\alpha(V,A(s))$ for any $s \in \Delta$.
\end{crl}

\begin{proof}
Since $Q(s)P(s) = - \res_\infty A(s)$ is invertible and $\alpha \neq 0$, 
we have
\[
\dim \Ker (P(s)Q(s)+\alpha\,1_W) = \dim \Ker (Q(s)P(s)+\alpha\,1_V) 
= \dim \Ker (L_0 + \alpha\,1_V).
\]
Hence the rank of the underlying bundle of the connection 
given by $\mc_\alpha(V,A(s))$ is equal to
\[
\rank (P(s)Q(s)+\alpha\,1_W) = \dim W - \dim \Ker (L_0 + \alpha\,1_V),
\]
which does not depend on $s \in \Delta$.
Fix a $\C$-vector space $V^\alpha$ with 
$\dim V^\alpha = \rank (P(s)Q(s)+\alpha\,1_W)$.
We can take an analytic family of linear maps 
\[
Q^\alpha(s) \colon W \to V^\alpha, 
\quad P^\alpha(s) \colon V^\alpha \to W, 
\quad s \in \Delta 
\]
so that for any $s \in \Delta$, 
$Q^\alpha(s)$ is injective, $P^\alpha(s)$ is surjective and 
$P^\alpha(s)Q^\alpha(s) = P(s)Q(s) + \alpha\,1_W$.
Set 
\[
A^\alpha(s) = Q^\alpha(s) (x 1_W - T(s))^{-1} P^\alpha(s)\, d_{\bbP^1}x, 
\quad s \in \Delta.
\]
Then we have $(V^\alpha, A^\alpha(s)) \simeq \mc_\alpha(V,A(s))$ 
for any $s \in \Delta$.
Since the meromorphic connection $\nabla^\vee$ is flat,
the connection 
\[
\nabla^\vee + \frac{\alpha\,1_W}{y}\,d_{\bbP^1}y = 
d_{\bbP^1 \times \Delta} + \left( T+\frac{PQ+\alpha\,1_W}{y}\right)d_{\bbP^1}y 
-(\Theta y + \Xi)
\]
is also flat.
Thus the arguments on condition (iii) in \thmref{thm:main} 
show that there exists a one-form $\Omega_\infty$ on $\Delta$ such that 
the connection 
$\nabla^\alpha := d_{\bbP^1 \times \Delta} 
- \Omega_\infty - Q^\alpha (x 1_W - T)^{-1}(d_{\bbP^1}x + \Theta)P^\alpha$ is flat.
\end{proof}

\begin{rmk}\label{rmk:admissible}
(i)~Equality~(12), Remarks~14,~16 in \cite{Yam11} 
and Theorems~5.1,~11.1 in \cite{Was} show that 
the family $(\nabla^\alpha_s)_{s \in \Delta}$, 
$\nabla^\alpha_s := d_{\bbP^1} - A^\alpha(s)$ 
introduced above is admissible if $\lambda + \alpha \notin \Z$ 
for all eigenvalues $\lambda$ of all $L^{(i)}_0$, $i=1,2, \dots , m$ and 
$L_0$ has no integral eigenvalue.

(ii)~For given isomonodromic family $(\nabla_s)_{s \in \Delta}$ 
with singularity data $(\bm{\Lambda},\bfL)$ satisfying $L_0=0$,
the family obtained by applying Hiroe's operator $\mathrm{mc}_{\mathbf{i}}$ (see \cite[Section~5.1]{Hir})
to $\nabla_s, s \in \Delta$ is admissible (and isomonodromic) if $\sum_{i=1}^m \lambda_i \not\in \Z$
for any tuple $(\lambda_i)_{i=1}^m$ with $\lambda_i$ an eigenvalue of $L^{(i)}_0$.
\end{rmk}

\providecommand{\bysame}{\leavevmode\hbox to3em{\hrulefill}\thinspace}


\end{document}